\documentclass[reqno, 11pt]{amsart}
\usepackage{amssymb,amsmath,amsfonts}
\usepackage[margin=1in]{geometry}
\usepackage{dsfont}
\usepackage{color}
\usepackage{graphicx}
\usepackage{latexsym}
\usepackage{amsmath}
\usepackage{amssymb}
\usepackage{graphics}
\usepackage[dvips]{epsfig}
\usepackage{mathrsfs}
\usepackage{mathtools}
\usepackage{leqno}
\usepackage{stmaryrd,cite}

\usepackage{cases}
\usepackage{enumerate}
\usepackage[usenames,dvipsnames,svgnames]{xcolor}
\definecolor{refkey}{rgb}{1,0,0.5}
\definecolor{labelkey}{rgb}{0,0.4,1}
\usepackage{todonotes}
\usepackage{marginnote}

\presetkeys{todonotes}{fancyline, color=green!40}{}

\makeatletter
\renewcommand{\@todonotes@drawMarginNoteWithLine}{%
	\begin{tikzpicture}[remember picture, overlay, baseline=-0.75ex]%
		\node [coordinate] (inText) {};%
	\end{tikzpicture}%
	\marginnote[{
		\@todonotes@drawMarginNote%
		\@todonotes@drawLineToLeftMargin%
	}]{
		\@todonotes@drawMarginNote%
		\@todonotes@drawLineToRightMargin%
	}%
}
\makeatother
\usepackage{listings}
\usepackage{ifpdf}
\ifpdf
\usepackage[CJKbookmarks=true,
hyperindex=true,
pdfstartview=FitH,
bookmarksnumbered=true,
bookmarksopen=true,
colorlinks=true,
citecolor=blue,
linkcolor=blue,
urlcolor=blue,
pdfborder=001,
pdfauthor={},
pdftitle={},
pdfkeywords={},
]{hyperref}
\else
\usepackage[
hypertex,
hyperindex,
linkcolor=blue,
unicode,
citecolor=blue%
]{hyperref}
\fi
\usepackage{cleveref}

\allowdisplaybreaks

\numberwithin{equation}{section}
\newtheorem{thm}{Theorem}[section]

\newtheorem{lem}[thm]{Lemma}
\newtheorem{prop}[thm]{Proposition}
\newtheorem{rmk}[thm]{Remark}

\newcommand{\pl}{\partial}

\newcommand{\be}{\begin{equation}}
	\newcommand{\ee}{\end{equation}}

\newcommand{\bee}{\begin{equation*}}
	\newcommand{\eee}{\end{equation*}}

\newcommand{\bse}{\begin{subequations}}
	\newcommand{\ese}{\end{subequations}}

\newcommand{\bs}{\begin{split}}
	\newcommand{\es}{\end{split}}

\begin{document}
	\author{Yingzhi Du$^{1}$}\thanks{$^{1}$Department of Mathematics, City University of Hong Kong, 83 Tat Chee Avenue, Kowloon Tong, Hong Kong.
		E-mail: yingzhidu2-c@my.cityu.edu.hk}

	\author{Hairong Liu$^{2}$}\thanks{$^{2}$ School of Mathematics and Statistics, Nanjing University of Science and Technology, Nanjing, 210094, P.R China
		E-mail: hrliu@njust.edu.cn}

	\title[] {Stability of steady states of the 3-D Navier-Stokes-Poisson equations with non-flat doping profile  in  exterior domains}

	\begin{abstract}
		This paper concerns an initial boundary value problem of compressible Navier-Stokes-Poisson equations with the non-flat doping profile in a 3-D exterior domain.
		The global existence of strong solutions near a steady state for compressible Navier-Stokes-Poisson equations with the general Navier-slip boundary conditions is established.  For our setting, not only a steady state should be constructed in the exterior domain  by the sub and super solution method, but also some new techniques should be adopted to obtain a priori estimates, especially some refined elliptic estimates and the estimates on the boundary.

		\end{abstract}
	
	\maketitle

	\section{Introduction}
	It is well-known that the compressible Navier-Stokes-Poisson (NSP) system consists of the Navier-Stokes equations coupled with the self-consistent Poisson equations, which is used in the simulation of the motion of charged particles (electrons or holes, see \cite{[23]} for more explanations). In three-dimensional space, the NSP system of one carrier type takes the following form
	\begin{equation}\label{origin}
		\left\{
		\begin{array}{lll}
			\rho_{t}+\mbox{div}(\rho u)=0,\\[2mm]
			\rho\left(u_{t}+(u\cdot\nabla)u\right)-\mu\Delta u-(\mu+\lambda)\nabla\mbox{div}u+\nabla p=\rho\nabla\Phi,\\[2mm]
			\Delta\Phi=\rho-b(x),
		\end{array}
		\right.
	\end{equation}
	where  $\rho>0$, $u=(u^1,u^2,u^3)$ and $p$ denote the density, the velocity field of charged particles, the pressure, respectively.
	The self-consistent electric potential $\Phi=\Phi(x,t)$ is coupled with the density through the Poisson equation.
	The pressure $p$ is expressed by
	\begin{equation}{\label{pressure}}
		p(\rho)=\rho^{\gamma},
	\end{equation}
	where $\gamma\geq1$ is a constant. As usual, the constant viscosity coefficients $\mu$ and $\lambda$ should satisfy the following physical conditions
	\begin{equation*}
		\mu>0,\,\,\,\lambda+\frac{2}{3}\mu\geq0.
	\end{equation*}
	$b=b(x)>0$ is the background profile, the sum of the background ion density, and the net density of impurities, which is assumed to be given.
	
	The object of this paper is to investigate the global existence  of the solutions to the initial boundary value (IBV) problem  of (\ref{origin}) in $(t,x)\in [0,+\infty)\times\Omega$ , where $\Omega\subset\mathbb{R}^3$ is
	a smooth exterior domain,  with the following initial and boundary conditions
	\begin{equation}{\label{initial}}
		(\rho, u)(x,t=0)=(\rho_0,u_0)(x),
	\end{equation}
	\begin{equation}\label{boundary1}
		u\cdot n=0,\quad 2(S(u)\cdot n)_{\tau}=-\alpha u_{\tau},\quad \mbox{on}\quad \partial\Omega,
	\end{equation}
and \begin{equation}\label{boundary3}
		\nabla\Phi\cdot n\mid_{\partial\Omega}=0,
	\end{equation}
	where $n$ stands for the outward unit normal to $\partial\Omega$.
	\eqref{boundary1} is the general Navier-slip boundary condition for the velocity field in which $2S(u)=\nabla u+(\nabla u)^{T}$ is the stress tensor,
	and $\alpha$ is a scalar friction function which measures the tendency of the fluid to slip on the boundary,
	$v_{\tau}$ stands for the tangential part of $v$ on $\partial\Omega$. It's worth mentioning  that $\alpha$ is non negative in usual. But mathematically, we can take into account the negative values of $\alpha$ as well.  For far-field, we impose the following conditions without
	loss of generality
	\begin{equation}\label{far}
		\Phi\rightarrow 0,\quad  \rho\rightarrow c_*, \quad \mbox{as}\quad |x|\rightarrow +\infty,
	\end{equation}
	for some positive constant $c_*$.

	{\bf Boundary condition \eqref{boundary1}  was}  proposed by Navier in \cite{Navierslip}, which allow the fluid to slip at the boundary, and have important applications for problems with rough boundaries, prorated boundaries, and interfacial boundary problems, see \cite{BG2008}.
	
	Note that the boundary condition (\ref{boundary1})  is
	equivalent to
	\begin{equation}\label{boundary2}
		u\cdot n=0,\quad \mbox{curl}u\times n=(2S(n)-\alpha I)u,\quad \mbox{on}\quad \partial\Omega,
	\end{equation}
	where $I$ is the $3\times3$ matrix.

	We now conduct a review of several key findings that are closely associated with the present research. From a mathematical perspective, the compressible Navier-Stokes-Poisson equations can be understood as a system that integrates the compressible Navier-Stokes equations with the Poisson equation.
	A wide array of studies on the Navier-Stokes equations has been undertaken by many mathematicians, driven by their considerable physical significance, diverse phenomena, and the mathematical challenges they pose.
	The global existence and asymptotic behavior of the solutions to the compressible Navier-Stokes equations in the whole space were studied extensively, see \cite{matsumura1980, MatsumuraNi1983, KK2002, KK2005, HoffZ1995, HoffZ1997, LW1998, I1971, CD2010, CD2015, L1996, FNP2001, HLX2012, LX2019, DUYZ2007, DLUY2007,  Deckelnick1993}.
	The recent work \cite{LX2019} from Li and Xin established global well-posedness and large-time asymptotic behavior of classical solutions to the system with vacuum in $\mathbb{R}^2$ or $\mathbb{R}^3$.
	In the context of an exterior domain $\Omega$ in $\mathbb{R}^3$, we refer to the works of \cite{KS1999, Kobayashi2002, J1996, Hoff2005, HS2008, HT2008, CLL, LLL}.
	Regarding the no-slip boundary condition $u|_{\pl \Omega}=0$, Matsumura and Nishida proved the global existence of smooth solutions near a constant state  away from vacuum  in the $L^2$ framework \cite{matsumura1982}. Subsequent studies have built upon these findings, establishing decay estimates and convergence rates for the solutions, as noted in \cite{KS1999, Kobayashi2002, J1996}.
	Concerning the special Navier-slip boundary conditions $(u \cdot n)|_{\pl \Omega}=0$, $(\mbox{curl} u \times n)|_{\pl \Omega}=Au$, Cai and  etc  established the global existence of classical solutions to the system in both bounded smooth domains \cite{CL} and exterior domains \cite{CLL}.

	We turn to related  results  regarding the global existence of solutions to the Cauchy problem for the Navier-Stokes-Poisson system.
In the case of a constant background doping profile, numerous results have been established. Related results on the  global existence of  smooth solutions to the Cauchy problem for the  compressible Navier-Stokes-Poisson  equations can be found in  \cite{Tan2013} for 1-D case  and \cite{LMAT2010} for multidimensional case. For further details, please refer to \cite{CD2017, HL2009, WW2010,  D1999, DFPS2004, GStrauss2005}.
Over these years, the asymptotic behavior on the system was studied  intensively, we refer to \cite{jang2013, wang, BWY2017, HL2010, HMW2003, LZ2012, WW2015, ZLZ2011}.
	In the case of a non-flat doping profile, there are relatively few results available, see \cite{FL2018, TWW2015, LXZ} and the references therein. Tan, Wang, and Wang established the global existence of classical solutions in the vicinity of the steady state \cite{TWW2015}. They also derived the time decay rates of the solution, contingent upon the initial perturbation being in $L^p$ for $1 \leq p < \frac{3}{2}$, provided that the doping profile remains small. 
Recently, Liu, Xu, and Zhang established the global well-posedness of strong solutions characterized by large oscillations and vacuum \cite{LXZ}, which is valid under the condition that the initial data has small energy and that the steady state is adequately separated from the vacuum.\\

	For the initial-boundary value problem of the compressible Navier-Stokes-Poisson system, numerous results pertain to a smooth bounded domain $\Omega$ in $\mathbb{R}^2$ or $\mathbb{R}^3$, with the boundary condition typically being the Dirichlet condition $u|_{\pl \Omega}=0$, see \cite{D2003, KS2008, ju2018, TW2009, TZ2010, LZ2022, shi2024, LLZ2020}. Donatelli established the local and global existence of weak solutions for this system with a non-constant background doping profile within the Sobolev framework \cite{D2003}. In \cite{KS2008}, Kobayashi and Suzuki demonstrated the global existence of a finite energy weak solution for this system with a flat doping profile in an appropriate function space. The local existence of strong solutions proved by Tan and Zhang \cite{TZ2010}, contingent upon the initial data fulfilling a natural compatibility condition. Recently, Liu and Zhong \cite{LZ2022} established the global existence of smooth solutions near a specified steady state for this system with a non-flat doping profile, demonstrating exponential stability. Their investigation permits significant variations in both the steady state and the background profile. In the case of unbounded domains, Liu, Luo, and Zhong \cite{LLZ2020} established the global existence of solutions to the compressible Navier-Stokes-Poisson equations with flat background profiles with Dirichlet boundary conditions. This significant result is applicable to large initial data situated in a domain exterior to a ball in $\mathbb{R}^3$, under the assumption of radial symmetry.

For the compressible  Navier-Stokes-Poisson equations of self-gravitating fluids modeling viscous gaseous stars, the vacuum free boundary problem is a very active research subject, for which the  gravitational force plays a different role compared with the electric forces. One may refer to \cite {jang2013,  luo2016,luoadv2016} for this topic.

 We make some comparison of the problem and results in this paper with those in  \cite{LLZ2020}. First, we note that, unlike in \cite{LLZ2020}, the background profile $b(x)$ in our paper is not constant, which means that the equilibrium (steady state) is no longer uniform. Consequently, our first objective is to establish the existence of steady state solutions. This is achieved by addressing an exterior problem associated with an elliptic equation, where we identify appropriate sub and super solutions to apply the comparison principle. Additionally, it is essential to derive the relevant $L^2$ estimates for the higher-order derivatives of the steady state solutions, as these estimates are critical for analyzing the long-term stability of the unsteady problem.
	
	Second, one notes that, a relaxation (damping) term is incorporated into the momentum equation in \cite{LLZ2020} when examining the exterior problem of the Navier-Stokes-Poisson equations in three dimensions without the symmetry assumption.  This relaxation term serves as a dissipation mechanism and is vital for the stability analysis in \cite{LLZ2020}. However, in this paper, we consider the scenario where such a relaxation term is absent from the momentum equation. Therefore, we must explore alternative approaches for obtaining stability estimates.

	For the Navier-slip boundary conditions governing the velocity field, it is more effective to utilize $L^2$ norms of divergence and curl to control the derivatives of the velocity. A significant observation is that the first Betti number of the exterior domain $\Omega =\mathbb{R}^3 \setminus \bar{U}$ is zero when $U$ is a bounded simply connected open set. In such scenarios, we apply the refined div-curl estimate (Proposition \ref{lem-div-curl}) to derive dissipation estimates for $\| \nabla u\|_{L^2}$, utilizing the estimates for $\| \mbox{curl} u \|_{L^2}$ and $\| \mbox{div} u\|_{L^2}$. This insight is crucial for obtaining various dissipation estimates.
	
	It is crucial to emphasize that proving the global existence of strong solutions to the initial boundary value problem under the boundary conditions (\ref{boundary1}) presents a significant challenge, particularly when compared to the homogeneous Dirichlet boundary conditions $u|_{\pl \Omega}=0$. In this context, precise estimates for $\mbox{div} u$ and $\mbox{curl} u$ are indispensable for the analysis undertaken in this paper.

	Now we state our main result.
	\begin{thm}\label{thm1}
		Let $\Omega=\mathbb{R}^3 \setminus B_R$ be a smooth exterior domain in $\mathbb{R}^3$, where $B_R$ is a ball with radius $R$. Assume  $b(x)>0$ and $\nabla b \in H^1(\Omega)$. Let $\tilde{ \rho}>0$, $\tilde{u} \equiv 0$ and $\tilde{\Phi}$ be a smooth steady state solution of (\ref{origin}) such that $\frac{\pl \tilde{\Phi}}{\pl n}|_{\pl \Omega}=0$. Then, for any given
 nonnegative function $\alpha$ in \eqref{boundary1},  there exists a constant $\delta>0$ such that if the initial data ($\rho_0$, $u_0$, $\Phi_0$) satisfy the compatibility condition with the boundary conditions,
		\begin{align}\label{condrho}
			\int_{\Omega} (\rho_0 - \tilde{ \rho}) \mathrm{d} x=0 ,
		\end{align}
		and
		\begin{align}
			(\|u\|_{H^3}+\|\rho-\tilde{ \rho}\|_{H^2}+\|(\rho_t,u_t)\|_{H^1}+\|\nabla(\Phi-\tilde{\Phi})\|_{L^2}+\|\nabla\Phi_t\|_{L^2} )(0) \leq \delta,
		\end{align}
		then there exists a unique global strong solution $(\rho, u ,\Phi)(t,x)$ to the initial boundary value problem (\ref{origin})-(\ref{far}) for all $t \geq 0$ satisfying
		\begin{align}
			&(\|u\|_{H^3}+\|\rho-\tilde{ \rho}\|_{H^2}+\|(\rho_t,u_t)\|_{H^1}+\|\nabla(\Phi-\tilde{\Phi})\|_{L^2}+\|\nabla\Phi_t\|_{L^2} )(t)\nonumber\\[2mm]
			&+ \int_{0}^{t} (\|\nabla u \|_{H^2}+\|\nabla u_{t}\|_{H^1}+\| \rho-\tilde{ \rho} \|_{H^2}+\| \rho_t\|_{H^1})(\tau) \mathrm{d} \tau\nonumber\\[2mm]
			&\leq C(\|u\|_{H^3}+\|\rho-\tilde{ \rho}\|_{H^2}+\|(\rho_t,u_t)\|_{H^1}+\|\nabla(\Phi-\tilde{\Phi})\|_{L^2}+\|\nabla\Phi_t\|_{L^2} )(0),
		\end{align}
		where $C$ is a positive constant independent of $t$.
	\end{thm}

	\begin{rmk}
		An important feature of this paper is that the profile $b(x)$ and steady state $\tilde{\rho}(x)$,  $\tilde{\Phi}(x)$ are allowed to be of large variation.
	\end{rmk}
	\begin{rmk}
		The condition (\ref{condrho}) persists in time, and is the necessary condition of solvability of the Poisson equation with Neumann boundary.
	\end{rmk}
	\begin{rmk}\label{rem1}
		The selection of $\Omega=\mathbb{R}^3 \setminus B_R$ is crucial for constructing the supersolution,  as in Lemma \ref{lem-rho} to obtain a steady state solution.Furthermore, for $\Omega=\mathbb{R}^3\setminus B_{R}$ the outward unit normal on its boundary $\partial\Omega$ is given by $n=-\left(\frac{x_1}{R},\frac{x_2}{R},\frac{x_3}{R}\right)$. It is straightforward to observe that the eigenvalues $S(n)$ are $0$, $-\frac{1}{R}$ and $-\frac{1}{R}$.
For the sake of simplicity, we assume that $\alpha$ is nonnegative. However, these conditions can be relaxed(see  Lemmas \ref{lem-basic}, \ref{lem-ut}.)
	\end{rmk}

\begin{rmk}
For the initial data close to the steady state $(\tilde{\rho},0,\tilde{\Phi})$ in higher-order Sobolev norm, we can
improve the regularity of the strong solutions in Theorem 1.1 to the smooth solutions.
Precisely, assume
\begin{align*}
			\mathcal{E}_1(0)&\equiv(\|u\|_{H^4}+\|\rho-\tilde{ \rho}\|_{H^3}+\|(\rho_t,u_t)\|_{H^2}+\|(q_{tt},u_{tt})\|_{L^2}\\[2mm]
&+\|\nabla(\Phi-\tilde{\Phi})\|_{L^2}+\|\nabla\Phi_t\|_{L^2}+\|\nabla\Phi_{tt}\|_{L^2})(0)\leq \delta,
		\end{align*}
		then there exists a unique global smooth solution $(\rho, u ,\Phi)(t,x)$ to the initial boundary value problem (\ref{origin})-(\ref{far}) for all $t \geq 0$ satisfying
		\begin{align*}
			\mathcal{E}_1(t)
			+ \int_{0}^{t} (\|\nabla u \|_{H^2}+\|\nabla u_{t}\|_{H^1}+\| \rho-\tilde{ \rho} \|_{H^3}+\| \rho_t\|_{H^2})(\tau) \mathrm{d} \tau
			\leq C \mathcal{E}_1(0),
		\end{align*}
		where $C$ is a positive constant independent of $t$.
\end{rmk}

	The rest of this paper is organized as follows. In Section 2, some useful elliptic estimates have been recalled firstly, we also prove some refined elliptic estimates for the Poisson equation with Neumann boundary in exterior domains. In Section 3, we
	construct a steady state solution $(\tilde{\rho}, \tilde{u}, \tilde{\Phi})(x)$ of (\ref{origin}) by the sub and super solution method, which helps us to reconstruct the problem for the perturbation variables $(q,u,\phi)(t,x)$.  In Section 4, we demonstrate the global existence of strong solutions to the initial boundary value problem for the Navier-Stokes-Poisson system.

	\section{Notation and basic lemmas}
	First, we introduce some notations. In this article, we denote the usual norm in the
	Sobolev space $H^{m}(\Omega)$ by $\|\cdot\|_{H^m}$, and the norm in $H^{m,p}$ by $\|\cdot\|_{H^{m,p}}$ where $m\geq0$,
	$p\geq1$; for $m=0$ we simply write $\|\cdot\|_{L^{p}}$.
	We recall some inequalities of the Sobolev type.
	\begin{lem}\label{sobolev inequ}
		Let $\Omega\subset\mathbb{R}^{3}$ be the exterior
		of a bounded domain with a smooth boundary. Then
		\begin{equation*}
			\begin{array}{lll}
				(i)\ \|f\|_{C^{0}(\bar{\Omega})}\leq C\|f\|_{H^{m,p}}\quad \mbox{for}\quad f\in H^{m,p}(\Omega),\quad mp>3>(m-1)p.\\[2mm]
				(ii)\ \|f\|_{L^{p}}\leq C\|f\|_{H^1}\quad \mbox{for}\quad f\in H^{1}(\Omega), \quad 2\leq p\leq 6.\\[2mm]
				(iii)\ \|f\|_{L^{6}}\leq C\|\nabla f\|_{L^2}\quad \mbox{for}\quad f\in H^{1}(\Omega).\\[2mm]
				(iv)\ \ \|f\|_{C^{0}(\bar{\Omega})}\leq C\|\nabla f\|_{H^1} \quad \mbox{for}\quad f\in H^{2}(\Omega).\\[2mm]
			\end{array}
		\end{equation*}
	\end{lem}

	The following lemma allows one to control the $H^{m}$-norm of a vector valued function $v$ by its
	$H^{m-1}$-norm of $\mbox{curl}v$ and $\mbox{div} v$ (see \cite{XX2007}).
	\begin{lem}\label{lem-div-curl}
		Let $\Omega$ be a domain in $\mathbb{R}^{3}$ with smooth boundary $\partial\Omega$ and outward normal $n$. Then there exists a constant $C>0$,
		such that
		\begin{equation}\label{A3-1}
			\|v\|_{H^{s}}\leq C\left(\|\mbox{div}v\|_{H^{s-1}}+\|\mbox{curl}v\|_{H^{s-1}}
			+ |v\cdot n|_{H^{s-\frac{1}{2}}(\partial\Omega)}+\|v\|_{L^2}\right),
		\end{equation}
		and
		\begin{equation}\label{A3-2}
			\|v\|_{H^{s}}\leq C\left(\|\mbox{div}v\|_{H^{s-1}}+\|\mbox{curl}v\|_{H^{s-1}}
			+ |v\times n|_{H^{s-\frac{1}{2}}(\partial\Omega)}+\|v\|_{L^2}\right),
		\end{equation}
		for any $v\in H^{s}(\Omega)$, $s\geq1$.
	\end{lem}
	
	In addition, since the first Betti number of the exterior domain $\Omega$ in our paper vanishes, we can apply Theorem 3.2 in \cite{wahl} to obtain the following refined estimate of $\|\nabla v\|_{L^2}$ which is crucial in the case of the exterior domain. We also note that the topological property of $\Omega=\mathbb{R}^3 \setminus B_{R}$ is necessary for the following Proposition.
	
	\begin{lem} (\cite{wahl} Theorem 3.2 )\label{prop-imp} Let $\nabla v\in L^{2}(\Omega)$, and $v\cdot n|_{\partial\Omega}=0$. The estimate ($C$ is a constant that independent of $v$)
		\begin{equation}\label{important}
			\|\nabla v\|_{L^2}\leq C\left(\|\mbox{div}v\|_{L^2}+\|\mbox{curl}v\|_{L^2}\right)
		\end{equation}
		is true for all $v$ as above if and only if $\Omega$ has a first Betti number of zero.
	\end{lem}

	It is important to mention that we introduce Lemma \ref{prop-imp} to eliminate the term $\|v\|_{L^2}$, as this term is challenging to manage in the context of the exterior domain. In the case of a bounded domain, however, Lemma \ref{lem-div-curl} suffices.

	\begin{lem}(\cite{CLL} Lemma 2.9)\label{lem-ell}
	Assume that $\Omega$ is an exterior domain of the simply connected domain $O$ in $\mathbb{R}^3$
 with smooth boundary.
		For any $q\in[2,4]$ there exists some positive constant $C=C(q,O)$ such that for every $v\in\{D^{1,2}|v(x)\rightarrow 0\quad\mbox{as}\quad |x|\rightarrow\infty\}$, it holds
		\begin{equation}\label{est-boundary}
			|v|_{L^{q}(\partial\Omega)}\leq C\|\nabla v\|_{L^2(\Omega)}.
		\end{equation}
		Moreover, for $k\geq1$, if $v\in\{D^{k+1,2}\cap D^{1,2}|v(x)\rightarrow 0\quad\mbox{as}\quad |x|\rightarrow\infty\}$
		with $v\cdot n|_{\partial\Omega}=0$ or $v\times n|_{\partial\Omega}=0$, then there exists some constant $C=C(k,D)$ such that
		\begin{equation}
			\|\nabla v\|_{H^{k}}\leq C\left(\|\mbox{div} v\|_{H^k}+\|\mbox{curl} v\|_{H^k}+\|\nabla v\|_{L^2}\right).
		\end{equation}
	\end{lem}
	
In order to  handle the general Navier-slip  boundary condition, we give the following  estimates about the boundary integrals.
	\begin{prop}\label{prop-trace}
		Let $\Omega=\mathbb{R}^3\setminus B_{R}$, then there exists some positive constant $C^{*}$   independent of $R$ and $v$, such that
		\begin{align*}
			|v|^2_{L^2(\partial \Omega)}\leq C^{*} R\|\nabla v\|^2_{L^2(\Omega)}.
		\end{align*}
	\end{prop}
	\begin{proof}
		Using the classical Sobolev trace estimate, we have
		\begin{align*}
			\int_{\partial B_1}v^2(y)d\sigma_{y}\leq C_1\int_{B_2\setminus B_1}v^2(y)dy+C_2\int_{B_2\setminus B_1}|\nabla v^2(y)|dy.
		\end{align*}
		By a scaling and H$\ddot{o}$lder inequality, it holds
		\begin{align*}
			\int_{\partial B_R}v^2(x)d\sigma_{x}&\leq C_1R^{-1}\int_{B_{2R}\setminus B_{R}}v^2(x)dx+C_2\int_{B_{2R}\setminus B_{R}}|\nabla v^2(x)|dx\\[2mm]
			&\leq C_1R \left(\int_{B_{2R}\setminus B_{R}}v^6(x)dx\right)^{\frac{1}{3}}
			+C_2R\left(\int_{B_{2R}\setminus B_{R}}v^6(x)dx\right)^{\frac{1}{6}}\left(\int_{B_{2R}\setminus B_{2R}}|\nabla v|^2dx\right)^{\frac{1}{2}}\\[2mm]
			&\leq C_1 R \left(\int_{\Omega}v^6(x)dx\right)^{\frac{1}{3}}
			+C_2R\left(\int_{\Omega}v^6(x)dx\right)^{\frac{1}{6}}\left(\int_{\Omega}|\nabla v|^2dx\right)^{\frac{1}{2}}\\[2mm]
			&\leq C R \|\nabla v\|^2_{L^2(\Omega)},
		\end{align*}
		where we have used $\|v\|_{L^6(\Omega)}\leq C\|\nabla v\|_{L^2(\Omega)}$.
	\end{proof}

\begin{prop}\label{lem-boundary-key}
Assume that $\Omega$ is an exterior domain of the bounded  domain in $\mathbb{R}^3$
 with smooth boundary. Let $\nabla v\in L^{2}(\Omega)$ with $v\cdot n=0$ on $\partial\Omega$. It
holds for $\nabla f \in L^{2}$,
\begin{align*}
\left|\int_{\partial\Omega}v\cdot\nabla f\right|\leq \|\nabla v\|_{L^2}\|\nabla f\|_{L^2}.
\end{align*}
\end{prop}
\begin{proof}
Denote $v^{\bot}=-v\times n$, it follows from the boundary condition $v\cdot n|_{\pl\Omega}=0$ that
\begin{equation*}
v=v^{\bot}\times n\quad \mbox{on}\quad \pl\Omega.
\end{equation*}
Therefore, by using some simple identities about vector fields and the divergence theorem, we
conclude that
\begin{align*}
\int_{\partial\Omega}v\cdot\nabla f&=\int_{\partial\Omega}v^{\bot}\times n\cdot\nabla f
=\int_{\partial\Omega}(\nabla f\times v^{\bot})\cdot n\\[2mm]
&=\int_{\Omega}\mbox{div}(\nabla f\times v^{\bot})
=-\int_{\Omega}\nabla f\times \mbox{curl}v^{\bot}\\[2mm]
&\leq \|\nabla f\|_{L^2}\|\nabla v\|_{L^2}.
\end{align*}
\end{proof}
	Next, we prove some elliptic estimates for elliptic equations with Navier-slip or Neumann boundary conditions in the exterior domain $\Omega$ by using Lemmas \ref{lem-div-curl}-\ref{lem-ell}, which will be frequently used in the proofs of the main results later.

\begin{lem} \label{lem-elliptic-1}
Assume that $\Omega$ is an exterior domain of the simply connected domain  in $\mathbb{R}^3$
 with smooth boundary, denote  $n$ by the outward normal of $\pl\Omega$.
Suppose that $u$ satisfy the Lam\'{e} equation
	\begin{equation}\label{equ-u}
		\left\{
		\begin{array}{llll}
			-\mu\Delta u-(\mu+\lambda)\nabla\mbox{div}u=g,\quad \mbox{in}\quad \Omega,\\[2mm]
			u\cdot n=0,\quad \mbox{curl} u\times n=K(x)u\quad \mbox{on}\quad \partial\Omega
		\end{array}
		\right.
	\end{equation}
There exists a positive constant $C$ depending only on $\mu,\lambda,\Omega$ and the matrix $K$ such that it holds
\begin{equation*}
	\|\nabla^2u\|_{H^{s}(\Omega)}\leq C\big(\|g\|_{H^{s}(\Omega)}+\| \nabla u\|_{L^{2}(\Omega)}\big),
\end{equation*}
where $s=0,1$.
\end{lem}
	
	Note the standard elliptic estimates for Lam\'{e} equation in bounded smooth domains can be seen in many literature. We give the proof in the setting of exterior domains here for the sake of completeness.
	
	\begin{proof}
		Firstly, applying Lemma \ref{lem-div-curl} with $v=\mbox{curl}u$ and using the boundary condition $\mbox{curl} u\times n=K(x)u$, we deduce
		\begin{align}\label{1}
			\|\mbox{curl}u\|_{H^1}&\leq C\left(\|\mbox{curl}^2u\|_{L^2}+|\mbox{curl}u\times n|_{H^{\frac{1}{2}}(\partial\Omega)}+\|\mbox{curl}u\|_{L^2}\right)\nonumber\\[2mm]
			&\leq C\left(\|\mbox{curl}^2u\|_{L^2}+|Ku|_{H^{\frac{1}{2}}(\partial\Omega)}+\|\mbox{curl}u\|_{L^2}\right)\nonumber\\[2mm]
			&\leq C\left(\|\nabla u\|^{1/2}_{L^2}\|\nabla^2 u\|^{1/2}_{L^2}+\|\mbox{curl}^2u\|_{L^2}+\|\mbox{curl}u\|_{L^2}\right)\nonumber\\[2mm]
			&\leq \varepsilon \|\nabla^2u\|_{L^2}+C\left(\|\nabla u\|_{L^2}+\|\mbox{curl}^2u\|_{L^2}\right),
		\end{align}
		for any $\varepsilon>0$.
		Moreover, applying Lemma \ref{lem-ell} with $v=u$, and using (\ref{1}), it holds
		\begin{align*}
			\|\nabla u\|_{H^1}&\leq C\left(\|\mbox{div}u\|_{H^1}+\|\mbox{curl} u\|_{H^1}+\|\nabla u\|_{L^2}\right)\nonumber\\[2mm]
			&\leq \varepsilon \|\nabla^2 u\|_{L^2}+C\left(\|\nabla\mbox{div}u\|_{L^2}+\|\mbox{curl}^2u\|_{L^2}+\|\nabla u\|_{L^2}\right).
		\end{align*}
		Therefore, we deduce that
		\begin{align}\label{2u}
			\|\nabla^2 u\|_{L^2}\leq C\left(\|\nabla\mbox{div}u\|_{L^2}+\|\mbox{curl}^2u\|_{L^2}+\|\nabla u\|_{L^2}\right),
		\end{align}
		by choosing $\varepsilon=\frac{1}{2}$.

		Next, we estimate $\|\nabla^3u\|_{L^2}$.  Applying Lemma \ref{lem-div-curl} with $v=\mbox{curl}u$ in (\ref{A3-2}) and $v=\mbox{curl}^2u$ in (\ref{A3-1}), we have
		\begin{align*}
			\|\mbox{curl} u\|_{H^2}&\leq C\left(\|\mbox{curl}^2u\|_{H^1}+|\mbox{curl}u\times n|_{H^{\frac{3}{2}}(\partial\Omega)}+\|\mbox{curl}u\|_{L^2}\right)\nonumber\\[2mm]
			&\leq C \left(\|\mbox{curl}^3u\|_{L^2}+|\mbox{curl}u^2\cdot n|_{H^{\frac{1}{2}}(\partial\Omega)}+\|\mbox{curl}^2u\|_{L^2}+|\mbox{curl}u\times n|_{H^{\frac{3}{2}}(\partial\Omega)}+\|\mbox{curl}u\|_{L^2}\right).
		\end{align*}
		Then, using Lemma \ref{lem-ell} with $v=u$, we get
		\begin{align}\label{uH3}
			&\|\nabla u\|_{H^2}\leq C\left(\|\mbox{div}u\|_{H^2}+\|\mbox{curl} u\|_{H^2}+\|\nabla u\|_{L^2}\right)\nonumber\\[2mm]
			&\leq C\left(\|\nabla^2\mbox{div}u\|_{L^2}+\|\mbox{curl}^3 u\|_{L^2}+\|\nabla u\|_{H^1}+|\mbox{curl}u^2\cdot n|_{H^{\frac{1}{2}}(\partial\Omega)}+|\mbox{curl}u\times n|_{H^{\frac{3}{2}}(\partial\Omega)}\right).
		\end{align}
		It remains to control the boundary integrals.

		Note that we can use the identities
		\begin{equation*}
			\nabla \cdot (\mbox{curl} u \times n) + \mbox{curl} u \cdot (\nabla \times n)= \mbox{curl}^2 u \cdot n,
		\end{equation*}
		and we can rewrite
		\begin{equation*}
			\mbox{div}(\mbox{curl} u \times n) = \pl_n (\mbox{curl} u \times n) \cdot n + (\Pi \pl_{i}(\mbox{curl} u \times n))^i,
		\end{equation*}
		with $\pl_n((\mbox{curl} u \times n) \cdot n)=0=\pl_n (\mbox{curl} u \times n) \cdot n + (\mbox{curl} u \times n) \cdot \pl_n n$ and $\Pi(\cdot)$ means the projection on vector fields tangent to the boundary that is $\Pi=Id-n \otimes n $. Therefore, we have
		\begin{equation*}
			|\mbox{curl}^2u\cdot n|_{H^{\frac{1}{2}}(\partial\Omega)} \leq C\left(|\mbox{curl} u|_{H^{\frac{1}{2}}(\partial\Omega)} + |K u|_{H^{\frac{3}{2}}(\partial\Omega)}\right)
		\end{equation*}
		from applying the boundary condition $(\mbox{curl}u \times n)|_{\pl \Omega}= Ku$.
		By using the trace theorem (\ref{est-boundary}), we obtain
		\begin{align*}
			|\mbox{curl}u\times n|_{H^{\frac{3}{2}}(\partial\Omega)}=|K u|_{H^{\frac{3}{2}}(\partial\Omega)}
			\leq \varepsilon \|\nabla^3 u\|_{L^2}+ C\left(\|\nabla u\|_{L^2}+\|\nabla^2 u\|_{L^2}\right),
		\end{align*}
		and
		\begin{align}\label{2ubd}
			|\mbox{curl}^2u\cdot n|_{H^{\frac{1}{2}}(\partial\Omega)}
			&\leq C(|\mbox{curl} u|_{H^{\frac{1}{2}}(\partial\Omega)} + |K u|_{H^{\frac{3}{2}}(\partial\Omega)})\nonumber\\[2mm]
			&\leq \varepsilon \|\nabla^3 u\|_{L^2}+C\left(\|\nabla u\|_{L^2}+\|\nabla^2 u\|_{L^2}\right).
		\end{align}
		Finally, putting the estimates of the boundary integrals into \eqref{uH3} and choosing $\varepsilon$ small enough, we get
		\begin{align}\label{est-u3}
			\|\nabla^3 u\|_{L^2}\leq C \left(\| \nabla^2 \mbox{div} u\|_{L^2}+\|\mbox{curl}^3u\|_{L^2}+\|\nabla^2 u\|_{L^2}+\|\nabla u\|_{L^2}\right).
		\end{align}

		With the help of  \eqref{2u} and \eqref{est-u3}, we should prove the elliptic estimates for the solutions of \eqref{equ-u}.
		Rewriting the Lam\'{e} equation as
		\begin{align}\label{section2-1}
			\mu\mbox{curl}^2u-(2\mu+\lambda)\nabla\mbox{div}u=g.
		\end{align}
		Multiplying  (\ref{section2-1}) by $\nabla\mbox{div}u$, it holds
		\begin{align*}
			 (2\mu+\lambda)\|\nabla\mbox{div}u\|_{L^2}^2=\mu\int_{\Omega}\mbox{curl}^2u\cdot\nabla\mbox{div}u-\int_{\Omega}g\cdot\nabla{div}u
		\end{align*}
		Integrating by parts with the boundary condition $(\mbox{curl}u \times n)|_{\partial\Omega}=K(x)u$, and applying Proposition \ref{lem-boundary-key} with $v=K(x)u$ and $f=\mbox{div} u$ to deal with the boundary term, one has
		\begin{align}\label{section2-3}
			\int_{\Omega}\mbox{curl}^2u\cdot\nabla\mbox{div}udx&=- \int_{\partial\Omega}(\mbox{curl}u\times n)\cdot\nabla\mbox{div}u
			=- \int_{\partial\Omega}(K(x)u)_{\tau}\cdot\nabla\mbox{div}u\nonumber\\[2mm]
			&\leq C\|\nabla\mbox{div}u\|_{L^2}\|\nabla(K(x)u)\|_{L^2}\leq \epsilon \|\nabla\mbox{div}u\|_{L^2}^2+C\|\nabla u\|_{L^2}^2,
		\end{align}
		which yields
		\begin{align}\label{section2-4}
			\|\nabla\mbox{div}u\|_{L^2}^2\leq C\|\nabla u\|_{L^2}^2+\|g\|_{L^2}^2.
		\end{align}
		Similarly, multiplying (\ref{section2-1}) by $\mbox{curl}^2u$,  and using (\ref{section2-3}) we have
		\begin{align}\label{section2-4-2}
			 \mu\|\mbox{curl}^2u\|_{L^2}^2&=\int_{\Omega}(2\mu+\lambda)\mbox{curl}^2u\cdot\nabla\mbox{div}u+\int_{\Omega}g\cdot\mbox{curl}^2u\nonumber\\[2mm]
			&\leq \epsilon \|\nabla\mbox{div}u\|_{L^2}^2+\frac{\mu}{2}\|\mbox{curl}^2u\|_{L^2}^2+C\|g\|_{L^2}^2.
		\end{align}
		Adding (\ref{section2-4}) and (\ref{section2-4-2}), and choosing $\epsilon$ small yields that
		\begin{align}\label{section2-5}
			\|\mbox{curl}^2u\|_{L^2}^2+\|\nabla\mbox{div}u\|_{L^2}^2 \leq C\left(\|\nabla u\|_{L^2}^2+\|g\|_{L^2}^2\right),
		\end{align}
		which together with  (\ref{2u})  immediately implies that
		\begin{align} \label{2ue}
			\|\nabla^2 u\|_{L^2}\leq C\left(\|\nabla u\|_{L^2}^2+\|g\|_{L^2}^2\right).
		\end{align}
		
Next, taking operator $\mbox{curl}$ to equation (\ref{section2-1}), then multiplying  the resulted equation  by $\mbox{curl}^3 u$, one obtains
		\begin{align}\label{section2-6}
			\mu \|\mbox{curl}^3 u \|^2_{L^2} = \int_{\Omega} \mbox{curl}g \cdot \mbox{curl}^3 u  \leq \frac{\mu}{2} \|\mbox{curl}^3 u \|^2_{L^2} + C \|\nabla g\|^2_{L^2},
		\end{align}
	which implies 	
\begin{equation*}
 \|\mbox{curl}^3 u \|^2_{L^2} \leq C \|\nabla g\|^2_{L^2}.
 \end{equation*}

Furthermore, computing  $\int_{\Omega}\nabla_{i} (\ref{section2-1})^j\cdot \nabla_{ij}\mbox{div}u$, and  applying Lemma \ref{lem-div-curl} and (\ref{section2-6}), (\ref{2ubd}), (\ref{2ue}), we get
		\begin{align*}
			(2\mu + \lambda) \|\nabla^2 \mbox{div} u \|^2_{L^2} &= \mu \int_{0}^{t} \nabla \mbox{curl}^2 u \cdot \nabla^2 \mbox{div}u - \nabla g \cdot  \nabla^2 \mbox{div}u \nonumber\\[2mm]
			&\leq \epsilon \|\nabla^2 \mbox{div}u \|^2_{L^2} + C (\|\nabla g\|^2_{L^2}+ \|\mbox{curl}^2 u \|^2_{H^1})\nonumber\\[2mm]
			&\leq \epsilon \|\nabla^2 \mbox{div}u \|^2_{L^2} + C (\|\nabla g\|^2_{L^2}+ \|\mbox{curl}^3 u \|^2_{L^2}+ |\mbox{curl}^2 u \cdot n|_{H^{\frac{1}{2}}(\pl \Omega)}+ \|\mbox{curl}^2 u \|^2_{L^2}),
		\end{align*}
		where we can choose $\epsilon=\mbox{min}\{\frac{2\mu+\lambda}{4}, \frac{\mu}{4}\}$ to get $\|\nabla^3 u  \|^2_{L^2} \leq C (\|g\|^2_{H^1}+\| \nabla u \|^2_{L^2})$, which ends the proof.
	\end{proof}

	\begin{lem}\label{lem-neu}
	Assume $\Omega$ is the same as in Lemma \ref{lem-elliptic-1}, and $\phi$ satisfies
		\begin{equation}\label{Neumann}
			\left\{
			\begin{array}{lll}
				\Delta \phi= q \quad \mbox{in}\quad \Omega\\[2mm]
				\nabla\phi\cdot n=0\quad \mbox{on}\quad \partial\Omega
			\end{array}
			\right.
		\end{equation}
		Then exists a positive constant $C$ depending only on $\Omega$ such that
		\begin{equation}\label{phi-1}
			\|\nabla^2\phi\|_{L^2}\leq C \|q\|_{L^2},
		\end{equation}
		\begin{equation}\label{phit}
			\|\nabla^2\phi_t\|_{L^2}\leq C \|q_t\|_{L^2},
		\end{equation}
		\begin{equation}\label{phitt}
			\|\nabla^{2}\phi_{tt}\|_{L^2}\leq C\|q_{tt}\|_{L^2},
		\end{equation}
		and
		\begin{align}\label{phi-3}
			\|\nabla^3\phi\|_{H^s}\leq C\left(\|\nabla q\|_{H^s}+\|q\|_{L^2}\right),
		\end{align}
		for $s=0,1$.
	\end{lem}

	Note the standard elliptic estimates for $\Delta \phi=q$ on the exterior domain  $\Omega$ with $\nabla\phi\cdot n|_{\partial\Omega}=0$:
	\begin{equation}\label{sobineq}
		\|\nabla^{m+2}\phi\|^2\leq C\left(\|q\|_m^2+\|\nabla\phi\|^2\right), \,\,\  m=0,1,2.
	\end{equation}
	One can refer to \cite{brezis} for proof. However, for the Navier-Stokes-Poisson system in exterior domains, $\|\nabla\phi\|_{L^2_t(L^2)}$ can't be obtained directly. To overcome this difficulty, we establish some refined elliptic estimates of $\phi$  here.
	\begin{proof}
		Applying Lemma \ref{prop-imp} to $\nabla\phi$, $\nabla \phi_{t}$, $\nabla\phi_{tt}$ with the boundary conditions $\nabla\phi\cdot n|_{\partial\Omega}=0$, $\nabla\phi_t\cdot n|_{\partial\Omega}=0$ and $\nabla\phi_{tt}\cdot n|_{\partial\Omega}=0$ respectively,  using the equation $\Delta\phi=q$, we can easily deduce that
		\begin{equation*}
			\|\nabla^2\phi\|_{L^2}\leq C \|\mbox{div}(\nabla \phi)\|_{L^2}=C\|\Delta \phi\|_{L^2}=C\|q\|_{L^2},
		\end{equation*}
		\begin{equation*}
			\|\nabla^2\phi_t\|_{L^2}\leq C \|\mbox{div}(\nabla \phi_t)\|_{L^2}=C\|\Delta \phi_t\|_{L^2 }=C\|q_t\|_{L^2},
		\end{equation*}
		and
		\begin{equation*}
			\|\nabla^{2}\phi_{tt}\|_{L^2}\leq C\|\Delta\phi_{tt}\|_{L^2}=C\|q_{tt}\|_{L^2}.
		\end{equation*}
		Moreover, let $v=\nabla\phi$, then (\ref{Neumann}) becomes
		\begin{equation*}
			\left\{
			\begin{array}{lll}
				\Delta v=\nabla q  \quad \mbox{in}\quad \Omega\\[2mm]
				v\cdot n=0,\quad \mbox{curl} v\times n=0\quad \mbox{on}\quad \partial\Omega
			\end{array}
			\right.
		\end{equation*}
		Therefore, applying Lemma \ref{lem-elliptic-1} to $v=\nabla\phi$, it holds
		\begin{align*}
			\|\nabla^3\phi\|_{H^s}&=\|\nabla^2v\|_{H^s}\leq C \left( \|\Delta v\|_{H^s}+\|\nabla v\|_{L^2}\right)\nonumber\\[2mm]
			&=C\left(\|\nabla (\Delta\phi)\|_{H^s}+\|\nabla^2\phi\|_{L^2}\right)\nonumber\\[2mm]
			&\leq C\left(\|\nabla q\|_{H^s}+\|q\|_{L^2}\right),
		\end{align*}
		due to (\ref{phi-1}).
	\end{proof}

	\section{Steady state}
	In this section, we   construct the steady state $(\tilde{\rho},0,\tilde{\Phi})$ of \eqref{origin}. A steady state with $u\equiv0$ must satisfy the following equations:
	\begin{equation}\label{rho-Phi-2}
		\Delta \tilde{\Phi}-\tilde{\rho}+b(x)=0, \quad \nabla \tilde{\Phi}=\gamma \tilde{\rho}^{\gamma-2}\nabla \tilde{\rho} \quad \mbox{in}\quad \Omega
	\end{equation}
	with the boundary condition $\frac{\partial\tilde{\Phi}}{\partial n}=0$ on $\partial\Omega$, and
	\begin{equation}\label{boundary}
		\tilde{\rho}(x)\rightarrow c_*,\quad \tilde{\Phi}(x)\rightarrow 0 \quad \mbox{as}
		\quad  |x|\rightarrow +\infty,
	\end{equation}
	for some positive constant $c_*$.

	\begin{rmk}\label{rem-steady}
		We would like to emphasize that, given our consideration of the exterior domain $\Omega=\mathbb{R}^3 \setminus B_R$ with the Navier-slip boundary condition, controlling the boundary term poses significant challenges compared to the analysis in \cite{LMAT2010}. The authors of \cite{LMAT2010} examined domains where $x \in \mathbb{R}^N$ for $N=2,3$, which facilitated obtaining higher-order derivative estimates without the need to address boundary terms. In our proof, we frequently utilize the relationships (\ref{rho-Phi-2}) between $\tilde{ \rho}$ and $\tilde{\Phi}$ to derive estimates for $\| \nabla \tilde{ \rho} \|^2_{H^3}$.
		Furthermore, the results from \cite{GStrauss2005} are not applicable in this context, as their findings regarding the existence of a unique smooth steady-state solution pertain specifically to bounded domains in $\mathbb{R}^3$.
	\end{rmk}
	
	\begin{lem}\label{lem-rho}
Let $\Omega=\mathbb{R}^3\backslash B_{R}$, and
assume that $b(x)$ satisfies
\begin{align}\label{cond-b2}
0<c_{*}\leq b(x)\leq c_{*}+r^{-1}.
\end{align}
Then there exists a unique solution $(\tilde{\rho},\nabla\tilde{\Phi})$ of (\ref{rho-Phi-2}) and (\ref{boundary}) under the case $1<\gamma\leq2$. Moreover, $\tilde{\rho}$ has the positive upper and lower bounds $c_{*}\leq \tilde{\rho}(x)\leq c_{*}+r^{-1}$,
and  there exists a constant $C$ depending on $\|\nabla b\|_{H^{1}}$, such that
\begin{align}\label{estimate-rho}
\|\nabla\tilde{\rho}\|^2_{H^3}\leq C,
\end{align}
if $\nabla b\in H^{1}(\Omega)$.
\end{lem}

\begin{proof}
For the case of $\gamma>1$, in virtue of \eqref{rho-Phi-2}, we take
\begin{equation}\label{rho}
\tilde{\rho}(x)=\left(\frac{\gamma-1}{\gamma}\right)^{\frac{1}{\gamma-1}}(\tilde{\Phi}+c_1)^{\frac{1}{\gamma-1}}  
\end{equation}
where $c_1= \frac{\gamma}{\gamma-1}c_{*}^{\gamma-1}$,
then (\ref{rho-Phi-2})
can be easily transformed into a semi-linear elliptic PDE of $\tilde{\Phi}$,
\begin{equation}\label{sequation1}
\left\{
\begin{array}{lll}
\Delta \tilde{\Phi}-\left(\frac{\gamma-1}{\gamma}\right)^{\frac{1}{\gamma-1}}(\tilde{\Phi}+c_{1})^{\frac{1}{\gamma-1}}+b(x)=0\quad \mbox{in} \quad \Omega,\\[2mm]
\frac{\partial\tilde{\Phi}}{\partial n}=0 \quad \mbox{on}\quad \partial\Omega,\\[2mm]
\tilde{\Phi}(x) \rightarrow 0 \quad \mbox{as}\quad |x|\rightarrow \infty.
\end{array}
\right.
\end{equation}
It is easy to see that $\tilde{\Phi}_1(x)=0$ is a subsolution of problem (\ref{sequation1}) since $b(x)\geq c_{*}$. We claim that  $\tilde{\Phi}_2(x)= \frac{\gamma}{\gamma-1} \left(c_{*}+r^{-1}\right)^{\gamma-1} - c_1$ is a supersolution of problem (\ref{sequation1}) for the case $1<\gamma\leq 2$.
Indeed, a direct compute gives
\begin{equation*}
\frac{\partial\tilde{\Phi}_2}{\partial r}=-\gamma r^{-2}\left(c_*+r^{-1}\right)^{\gamma-2},
\end{equation*}
and
\begin{equation*}
\Delta \tilde{\Phi}_2=\gamma(\gamma-2)r^{-4}\left(c_{*}+r^{-1}\right)^{\gamma-3}.
\end{equation*}
Then
\begin{equation*}
\frac{\partial\tilde{\Phi}_2}{\partial n}\Big|_{\partial\Omega}=-\frac{\partial\tilde{\Phi}_2}{\partial r}\Big|_{r=R}>0,
\end{equation*}
and
\begin{align*}
\Delta \tilde{\Phi}_2-\left(\frac{\gamma-1}{\gamma}\right)^{\frac{1}{\gamma-1}}(\tilde{\Phi}_2+c_{1})^{\frac{1}{\gamma-1}}+b(x)
\leq
\gamma(\gamma-2)r^{-4}(c_{*}+r^{-1})^{\gamma-3}
\leq 0,
\end{align*}
thanks to $b(x)\leq c_{*}+r^{-1}$ and $1<\gamma\leq 2$.
According to the results in \cite{N1979}, there exists a solution $\tilde{\Phi}$ of (\ref{sequation1}) with $0\leq\tilde{\Phi}\leq (c_{*}+r^{-1})^{\gamma-1} \frac{\gamma}{\gamma-1} - c_1$.  Therefore, there exists a  solution  of $(\tilde{\rho},\nabla\tilde{\Phi})$ of (\ref{rho-Phi-2}), (\ref{boundary}),
with
\begin{equation*}
c_*\leq \tilde{\rho}\leq c_*+r^{-1},
\end{equation*}
due to (\ref{rho}).

For the case of $\gamma=1$, we take
\begin{equation*}
\tilde{\rho}(x)=c_*e^{\tilde{\Phi}(x)},
\end{equation*}
then   (\ref{rho-Phi-2}) becomes
\begin{equation}\label{sequation2}
\left\{
\begin{array}{lll}
\Delta \tilde{\Phi}-c_*e^{\tilde{\Phi}}+b(x)=0,\quad \mbox{in} \quad \Omega\\[2mm]
\frac{\partial\tilde{\Phi}}{\partial n}=0 \quad \mbox{on}\quad \partial\Omega\\[2mm]
\tilde{\Phi}(x) \rightarrow 0 \quad \mbox{as}\quad |x|\rightarrow \infty.
\end{array}
\right.
\end{equation}

Similar as above, we can easily check that $\tilde{\Phi}_3(x)=0$ is a subsolution  and $\tilde{\Phi}_4(x)=\ln (1+c_*^{-1}r^{-1})$
is a supersolution of (\ref{sequation2}) by our assumption that $0< c_{*} \leq b(x) \leq c_{*}+r^{-1}$. Therefore,
there exists a solution $\tilde{\Phi}$ of (\ref{sequation2}) with $0 \leq \tilde{\Phi} \leq \ln (1+c_{*}^{-1}r^{-1})$.
Then, the steady system  (\ref{rho-Phi-2}), (\ref{boundary}) exists a solution $(\tilde{\Phi},\tilde{\rho})$ with   $c_*\leq \tilde{ \rho}\leq c_*+r^{-1}$.

	Next, we shall prove some estimates of $\tilde{\rho}$, which will be frequently used in the following sections. Rewriting the steady problem as
		\begin{equation}\label{equ}
			\left\{
			\begin{array}{llll}
				\Delta \tilde{\Phi}=\tilde{\rho}-b(x),\quad \mbox{in} \quad \Omega\\[2mm]
				\frac{\partial\tilde{\Phi}}{\partial n}=0 \quad \mbox{on}\quad \partial\Omega\\[2mm]
				\tilde{\Phi}(x) \rightarrow 0 \quad \mbox{as}\quad |x|\rightarrow \infty.
			\end{array}
			\right.
		\end{equation}
		Multiplying  the equation (\ref{equ}) by $\Delta\tilde{\Phi}$, after integrating by parts with the homogeneous Neumann boundary condition $\frac{\pl \tilde{\Phi}}{\pl n}|_{\pl \Omega}=0$, it holds that
		\begin{equation*}
			\int_{\Omega}|\Delta \tilde{\Phi}|^2+\int_{\Omega}\nabla\left(\tilde{\rho}-b\right)\cdot\nabla\tilde{\Phi}=0.
		\end{equation*}
		In virtue of (\ref{rho-Phi-2}), we obtain
		\begin{equation}\label{ineq-1}
			\int_{\Omega}|\Delta \tilde{\Phi}|^2+\int_{\Omega}\gamma \widetilde{\rho}^{\gamma-2}\nabla\left(\tilde{\rho}-b\right)\cdot\nabla\tilde{\rho}=0.
		\end{equation}
		Using the boundedness of $\tilde{\rho}$ that we have already proved and the Cauchy's inequality, (\ref{ineq-1}) implies that
		\begin{equation}\label{ineq-2}
			\|\Delta \tilde{\Phi}\|_{L^2}^2+\|\nabla \left(\tilde{\rho}-b\right)\|_{L^2}^2\leq C\|\nabla b\|_{L^2}^2.
		\end{equation}
		Moreover, using (\ref{rho-Phi-2}) and the boundedness of $\tilde{ \rho}$ again, one has
		\begin{equation}\label{ineq-3}
			\|\nabla\tilde{\Phi}\|_{L^2}^2\leq C\|\nabla\tilde{\rho}\|^2_{L^2}\leq C\|\nabla b\|^2_{L^2},
		\end{equation}
		for the last inequality we have used $\|\nabla\tilde{\rho}\|^2_{L^2} \leq C (\|\nabla(\tilde{ \rho}-b)\|^2_{L^2} + \|\nabla b\|^2_{L^2})$ and (\ref{ineq-2}).
		
		Applying Lemma \ref{lem-neu} with $\phi=\tilde{\Phi}$, using (\ref{ineq-2}), we have
		\begin{equation}\label{ineq-4}
			\|\nabla^2\tilde{\Phi}\|^2_{L^2}\leq C \|\Delta\tilde{\Phi}\|^2_{L^2} \leq C\|\nabla b\|^2_{L^2},
		\end{equation}
		\begin{equation}\label{ineq-4-2}
			\|\nabla^3\tilde{\Phi}\|^2_{L^2}\leq C\|\Delta\tilde{\Phi}\|^2_{H^1}
			\leq C\left(\|\Delta\tilde{\Phi}\|^2_{L^2}+\|\nabla(\tilde{\rho}-b)\|^2_{L^2}\right)\leq C\|\nabla b\|^2_{L^2}.
		\end{equation}
		Furthermore, applying Sobolev inequality and  (\ref{ineq-3}) and (\ref{ineq-4-2}), one gets
		\begin{equation}\label{ineq-5}
			\|\nabla\tilde{\Phi}\|_{L^{\infty}}\leq \|\nabla^3\tilde{\Phi}\|^{\frac{3}{4}}_{L^2}\|\nabla\tilde{\Phi}\|^{\frac{1}{4}}_{L^2}\leq C\|\nabla b\|_{L^2}.
		\end{equation}

		A direct calculation gives
		\begin{equation*}
			 \partial_{ij}\tilde{\rho}=\tilde{\rho}\left(\partial_{ij}\tilde{\Phi}+\partial_{i}\tilde{\Phi}\partial_{j}\tilde{\Phi}\right), \quad \mbox{if}\quad \gamma=1,
		\end{equation*}
		and
		\begin{equation*}
			\partial_{ij}\tilde{\rho}=C_{1}(\gamma)\tilde{\rho}^{2-\gamma}\partial_{ij}\tilde{\Phi}
			+C_{2}(\gamma)\tilde{\rho}^{3-2\gamma}\partial_{i}\tilde{\Phi}\partial_{j}\tilde{\Phi}, \quad \mbox{if}\quad \gamma>1,
		\end{equation*}
		note here we use $C(\cdot)$ to represent the quantities that $C$ relies on.
		In virtue of the  positive upper and
		lower bounds of $\tilde{\rho}$, it holds that
		\begin{align}\label{ineq-6}
			\|\nabla^2\tilde{\rho}\|^2_{L^2}&\leq C\left(\|\nabla^2\tilde{\Phi}\|^2_{L^2}+\|\nabla\tilde{\Phi}\|^4_{L^4}\right)\nonumber\\
			&\leq C\left(\|\nabla^2\tilde{\Phi}\|^2_{L^2}+\|\nabla\tilde{\Phi}\|^2_{L^{\infty}}\|\nabla\tilde{\Phi}\|^2_{L^2}\right)\nonumber\\
			&\leq C\left(\|\nabla b\|^2_{L^2}+\|\nabla b\|^4_{L^2}\right),
		\end{align}
		where we used (\ref{ineq-3}), (\ref{ineq-4}) and (\ref{ineq-5}).
		
		Similarly, the third-order derivation of $\tilde{\rho}$ is
		\begin{equation*}
			 \partial_{ijk}\tilde{\rho}=\tilde{\rho}\left(\partial_{ijk}\tilde{\Phi}+\partial_{ik}\tilde{\Phi}\partial_{j}\tilde{\Phi}
			+\partial_{k}\tilde{\Phi}\partial_{jk}\tilde{\Phi}\right)+\tilde{\rho}\partial_{k}\tilde{\Phi}
			\left(\partial_{ij}\tilde{\Phi}+\partial_{i}\tilde{\Phi}\partial_{j}\tilde{\Phi}\right), \quad \mbox{if}\quad \gamma=1,
		\end{equation*}
		and
		\begin{equation*}
			\partial_{ijk}\tilde{\rho}=C_{3}(\gamma)\tilde{\rho}^{2-\gamma}\partial_{ijk}\tilde{\Phi}
			+C_{4}(\gamma)\tilde{\rho}^{3-2\gamma}\partial_{ik}\tilde{\Phi}\partial_{j}\tilde{\Phi}
			+C_{5}(\gamma)\tilde{\rho}^{1-\gamma}\partial_{ij}\tilde{\Phi}\partial_{k}\tilde{\Phi}
			 +C_{6}(\gamma)\tilde{\rho}^{2-2\gamma}\partial_{i}\tilde{\Phi}\partial_{j}\tilde{\Phi}\partial_{k}\tilde{\Phi}
			, \quad \mbox{if}\quad \gamma>1.
		\end{equation*}
		Therefore,
		\begin{align} \label{ineq-7}
			\|\nabla^3\tilde{\rho}\|^2_{L^2}
			&\leq C\left(\|\nabla^3\tilde{\Phi}\|^2_{L^2}+\|\nabla\tilde{\Phi}\|^2_{L^{\infty}}\|\nabla^2\tilde{\Phi}\|^2_{L^2}
			+\|\nabla^{2}\tilde{\Phi}\|_{L^2}^6\right)\nonumber\\[2mm]
			&\leq C\left(\|\nabla b\|^2_{L^2}+\|\nabla b\|^4_{L^2}+\|\nabla b\|^6_{L^2}\right),
		\end{align}
		following the same argument as (\ref{ineq-6}).
		
		To get higher order estimates, we first consider
		\begin{align}
			\|\nabla^4\tilde{\Phi}\|^2_{L^2}&\leq C\|\Delta\tilde{\Phi}\|^2_{H^2} \leq C\| \nabla^2(\tilde{ \rho}-b)\|^2_{L^2}\nonumber\\[2mm]
			&\leq C\left(\|\nabla^2 \tilde{ \rho} \|^2_{L^2}+\|\nabla^2b\|^2_{L^2}\right)\nonumber\\[2mm]
			&\leq C \left(\|\nabla b\|^4_{L^2}+\|\nabla b\|^2_{L^2}+ \|\nabla^2 b\|^2_{L^2}\right),\nonumber
		\end{align}
		and then
		\begin{align}
			\|\nabla^2\tilde{\Phi}\|^2_{L^{\infty}}\leq C(\|\nabla b\|_{L^2}, \|\nabla^2 b\|_{L^2}).\nonumber
		\end{align}

		By omitting $C(\gamma)$ and $\tilde{ \rho}$ for simplicity since they are bounded, we get the forth-order derivation of $\tilde{\rho}$
		\begin{align}
			\nabla^4\tilde{\rho}\sim \nabla^4\tilde{\Phi}+\nabla^3\tilde{\Phi}\nabla\tilde{\Phi}
			+(\nabla^2\tilde{\Phi})^2+\nabla^2\tilde{\Phi}(\nabla\tilde{\Phi})^2+(\nabla\tilde{\Phi})^4.
		\end{align}
		Therefore, we get
		\begin{align}
			&\|\nabla^4\tilde{\rho}\|^2_{L^2}\leq C\int_{\Omega}|\nabla^4\Phi|^2+\int_{\Omega}|\nabla^3\tilde{\Phi}|^2|\nabla\tilde{\Phi}|^2
			 +|\nabla^2\tilde{\Phi}|^4+|\nabla^2\tilde{\Phi}|^2|\nabla\tilde{\Phi}|^4+|\nabla\tilde{\Phi}|^8\nonumber\\[2mm]
			&\leq  C\left(\|\nabla^4\Phi\|^2_{L^2}+\|\nabla\Phi\|^2_{L^{\infty}}\|\nabla^3\Phi\|^2_{L^2}
			+\|\nabla^2\Phi\|^2_{L^{\infty}}\|\nabla^2\Phi\|^2_{L^2}
			+\|\nabla\Phi\|^4_{L^{\infty}}\|\nabla^2 \Phi\|_{L^2}^2
			+\|\nabla\Phi\|^2_{L^{\infty}}\|\nabla \Phi\|_{L^6}^6\right)\nonumber\\
			&\leq C\left(\|\nabla b\|_{L^2}, \|\nabla^2 b\|_{L^2}\right).\nonumber
		\end{align}
		This completes the proof.
	\end{proof}
	
	{\begin{rmk}For general case $\gamma>1$, assume  $b(x)$ satisfies
\begin{align}\label{cond-b2}
0<c_{*}\leq b(x)\leq  \left(\frac{\gamma-1}{\gamma}\right)^{\frac{1}{\gamma-1}} \left(c_0r^{-\varepsilon}+\frac{\gamma}{\gamma-1}c_{*}^{\gamma-1}\right)^{\frac{1}{\gamma-1}},
\end{align}
for some positive constant $c_0$ and $0<\varepsilon<1$.
We can also prove the existence of a steady state solution $(\tilde{\rho},\tilde{\Phi})$  by constructing the following supersolution $\tilde{\Phi}_2(x)=c_0r^{-\varepsilon}$.  Moreover, $\tilde{\rho}$ has the positive upper and lower bounds
\begin{equation*}
c_{*}\leq \tilde{\rho}(x)\leq \left(\frac{\gamma-1}{\gamma}\right)^{\frac{1}{\gamma-1}} \left(c_0r^{-\varepsilon}+\frac{\gamma}{\gamma-1}c_{*}^{\gamma-1}\right)^{\frac{1}{\gamma-1}}.
\end{equation*}
\end{rmk}
	\vspace{2mm}

	\section{Global existence of strong solutions}
	
	Let $q=\rho-\tilde{\rho}$, $\phi=\Phi-\tilde{\Phi}$. In order to reformulate the problem (\ref{origin})-(\ref{far}) properly, we introduce the enthalpy function
	\begin{equation*}
		h'(s)=\frac{p'(s)}{s}.
	\end{equation*}
	Then the initial boundary value problem (\ref{origin})-(\ref{boundary1}) can be reformulated into the perturbed form of
	\begin{equation}\label{equation-1}
		\left\{
		\begin{array}{lll}
			q_{t}+\mbox{div}(\tilde{\rho} u)=-\mbox{div}(q u)\equiv f^{0},\\[2mm]
			u_{t}+\nabla (h'(\tilde{\rho})q)-\tilde{\rho}^{-1}\left(\mu\Delta u+(\mu+\lambda)\nabla \mbox{div}u\right)-\nabla\phi=f,\\[2mm]
			\Delta\phi=q,\\[2mm]
			u\cdot n|_{\partial\Omega}=0,\quad \mbox{curl}u\times n|_{\partial\Omega}=(2S(n)-\alpha I)u,\\[2mm]
			\frac{\partial\phi}{\partial n}|_{\partial\Omega}=0,\\[2mm]
			u(0,\cdot)=u_0,\quad q(0,\cdot)=q_0\equiv \rho_0-\tilde{\rho},
		\end{array}
		\right.
	\end{equation}
	where the nonlinear term on $\eqref{equation-1}_2$ is described as
	\begin{equation}\label{eq-f}
		f=-u\cdot \nabla u + \gamma \nabla (q+\tilde{ \rho}) (\tilde{ \rho}^{\gamma-2} + (q+\tilde{ \rho})^{\gamma-2}) + \gamma (\gamma-2) \nabla \tilde{ \rho} \tilde{ \rho}^{\gamma-3} q +\left(\frac{1}{q+\tilde{\rho}}-\frac{1}{\tilde{\rho}}\right)\left(\mu\Delta u+(\mu+\lambda)\nabla \mbox{div}u\right).
	\end{equation}
	
	Theorem \ref{thm1} will be proved in this section. The local-in-time well-posedness of problem  \eqref{equation-1} can be proved by using the
	linearization and iteration technique, following the arguments in \cite{Ma1984}; therefore, to prove Theorem \ref{thm1}, it suffices to prove the following a priori estimates.
	For clarity, we introduce  quantities $\mathcal{E}(t)$ and $\mathcal{D}(t)$ as follows:
	\begin{equation}\label{E3}
		\mathcal{E}(t)=\|u(\cdot, t)\|_{H^3}+\|q(\cdot, t)\|_{H^2}+\|(q_t,u_t)(\cdot, t)\|_{H^1}+\|\nabla\phi(\cdot,t)\|_{L^2}+\|\nabla\phi_t(\cdot,t)\|_{L^2},
	\end{equation}
	and
	\begin{equation}\label{D3}
		\begin{split}
			\mathcal{D}(t)=\|\nabla u(\cdot, t)\|_{H^2}+\|\nabla u_{t}(\cdot, t)\|_{H^1}+\| q(\cdot, t)\|_{H^2}+\| q_t(\cdot, t)\|_{H^1}+\| q_{tt}(\cdot, t)\|_{L^2}.
		\end{split}
	\end{equation}

	\begin{prop}\label{prop} ({\it {a priori estimates}})
		Suppose that for some $T>0$, $(q,u,\phi)$ is a solution to the initial boundary value problem (\ref{equation-1}) in the time interval $t\in[0,T]$. Then there exist  positive constants $\delta$, $c$ and  $C$  which are independent of $t$,  such that if
		\begin{equation}\label{small}
			\sup_{0\leq t\leq T}\mathcal{E}(t)\leq \delta,
		\end{equation}
		then there holds, for any $t\in[0,T]$,
		\begin{equation*}
			\mathcal{E}^2(t)+c\int_{0}^{t}\mathcal{D}^2(s)ds\leq C\mathcal{E}^2(0).
		\end{equation*}
	\end{prop}
	
	In the following, we will derive the a priori estimates for the solutions to (\ref{equation-1}) under the assumption  \eqref{small}
	for sufficiently small $\delta>0$.

	We want to remark that according to the definition of $h'$, it is easy to see that
	\begin{align*}
		h'(\tilde{\rho})=\gamma\tilde{\rho}^{\gamma-2},
	\end{align*}
	thus by the boundedness of $\tilde{ \rho}$ and $\|\nabla \tilde{\rho}\|_{H^3}$, there exist some positive constants $C_1, C_2, C$ such that
	\begin{align}\label{hb}
		0<C_1<h'(\tilde{\rho})<C_2, \quad |h''(\tilde{\rho})|\leq C, \quad \mbox{and} \quad  |h'''(\tilde{\rho})|\leq C.
	\end{align}
	These estimates will be used frequently in the following proofs.

	\subsection{Estimates for lower-order derivatives}
	
	We first derive the zero-order energy estimates for the solution itself.
	\begin{lem}\label{lem-basic}
		Suppose that the conditions in Proposition \ref{prop} hold, then for the solutions $(q,u,\nabla\phi)$ of (\ref{equation-1}) in any exterior domains $\Omega=\mathbb{R}^3\setminus B_{R}$, it holds
		\begin{equation}\label{result-1}
			\frac{1}{2}\frac{d}{dt}\int_{\Omega}\left(\tilde{\rho}|u|^2+h'(\tilde{\rho})|q|^2+|\nabla \phi|^2\right)+c\|\nabla u\|_{L^2}^2 \leq C\delta
			\mathcal{D}^2(t),
		\end{equation}
		provided $\alpha>0$, where $c$ is a constant only depends on $\mu$ and $\lambda$ and constant $C>0$ independent of $t$.
		For general case $\alpha$, then the energy estimate (\ref{result-1}) is also valid for exterior domains $\Omega=\mathbb{R}^3\setminus B_{R}$ with small $R$.
	\end{lem}
	\begin{proof}
		Rewriting the second equation of (\ref{equation-1}) by using the identity $\Delta u=-\mbox{curl}^2u+\nabla\mbox{div}u$, we have
		\begin{equation}\label{equation-u}
			u_{t}+\nabla (h'(\tilde{\rho})q)+\mu\tilde{\rho}^{-1}\mbox{curl}^2 u-(2\mu+\lambda)\tilde{\rho}^{-1}\nabla \mbox{div}u-\nabla\phi=f.
		\end{equation}
		Then computing the following integral
		\begin{equation*}
			\int_{\Omega}\eqref{equation-1}_1h'(\tilde{\rho})q+\eqref{equation-u}\cdot\tilde{\rho} u,
		\end{equation*}
		it holds
		\begin{align}\label{step1}
			&\int_{\Omega}\left(\tilde{\rho} u_t\cdot u+h'(\tilde{\rho})q_tq\right)
			+\int_{\Omega}\left(\mbox{div}(\tilde{\rho} u)h'(\tilde{\rho})q+\tilde{\rho} u\cdot\nabla(h'(\tilde{\rho})q)\right)\nonumber\\[2mm]
			&+\mu\int_{\Omega} \mbox{curl}^2u\cdot u-(2\mu+\lambda)\int_{\Omega} \nabla\mbox{div}u\cdot u
			-\int_{\Omega}\tilde{ \rho} u\cdot\nabla\phi
			=\int_{\Omega}h'(\tilde{\rho})f^0q+\tilde{\rho} f\cdot u.
		\end{align}

		Integrating by parts with the boundary condition $(u\cdot n)|_{\pl \Omega}=0$, one can easily check
		\begin{align*}
			\int_{\Omega}\left(\mbox{div}(\tilde{\rho} u)h'(\tilde{\rho})q+\tilde{\rho} u\cdot\nabla(h'(\tilde{\rho})q)\right)=0,
		\end{align*}
		and
		\begin{align*}
			-\int_{\Omega} \nabla\mbox{div}u\cdot u=\int_{\Omega} |\mbox{div}u|^2.
		\end{align*}

		For the term involving $\mbox{curl}u$ on the left-hand side of (\ref{step1}), we do integration by parts to get
		\begin{align}
			\int_{\Omega}\mbox{curl}^2u \cdot u&= \int_{\Omega}|\mbox{curl}u|^2-\int_{\partial\Omega}(\mbox{curl} u\times n)\cdot u\nonumber\\[2mm]
			&=\|\mbox{curl}u\|^2_{L^2}-2\int_{\partial\Omega}S(n)u\cdot u+\int_{\partial\Omega} \alpha^{+}|u|^2-\int_{\partial\Omega}\alpha^{-} |u|^2,
		\end{align}
		due to the boundary condition $\mbox{curl} u\times n=(2S(n)-\alpha I)u$, where  $\alpha^{+}=\sup\left\{\alpha,0\right\}$ and $\alpha^{-}=\sup\{-\alpha,0\}$.
		Noting that $-S(n)$ is positive semi-definite in our setting $\Omega\equiv\mathbb{R}^3\setminus B_R$  (see Remark \ref{rem1}), then the above boundary integrals $-\int_{\partial\Omega}S(n)u\cdot u+\int_{\partial\Omega} \alpha^{+}|u|^2$ is good terms in the energy estimate since they are nonnegative terms on the left-hand side. Then we shall use the trace theorem to estimate the boundary integral $\int_{\partial\Omega} \alpha^{-} |u|^2$
		\begin{align*}
			\int_{\partial\Omega} \alpha^{-} |u|^2\leq |\alpha^{-}|_{L^{\infty}}|u|^2_{L^{2}(\partial\Omega)}\leq C^{*}|\alpha^{-}|_{L^{\infty}} R \|\nabla u\|^2_{L^2},
		\end{align*}
		where we used Proposition \ref{prop-trace}.

		To estimate the last term on the left-hand side of (\ref{step1}),  integrating by parts with the boundary conditions $(u\cdot n)|_{\partial\Omega}=0$, $(\nabla\phi\cdot n)|_{\partial\Omega}=0$, and using $q_t + \mbox{div}(\tilde{ \rho} u)=-\mbox{div}(qu)$ and the Poisson equation $\Delta \phi =q$ in (\ref{equation-1}), we obtain
		\begin{align*}
			-\int_{\Omega}\tilde{\rho}\nabla\phi\cdot u&=\int_{\Omega}\phi\mbox{div}(\tilde{\rho}u)
			=-\int_{\Omega}\phi q_{t}-\int_{\Omega}\phi\mbox{div}(qu)\nonumber\\[2mm]
			&=-\int_{\Omega}\phi \Delta\phi_{t}+\int_{\Omega}qu\cdot\nabla\phi\nonumber\\[2mm]
			&=\frac{1}{2}\frac{d}{dt}\int_{\Omega}|\nabla\phi|^2+\int_{\Omega}qu\cdot\nabla\phi\nonumber\\[2mm]
			&\geq \frac{1}{2}\frac{d}{dt}\int_{\Omega}|\nabla\phi|^2-\|u\|_{L^3}\|q\|_{L^2}\|\nabla\phi\|_{L^6}\nonumber\\[2mm]
			&\geq \frac{1}{2}\frac{d}{dt}\int_{\Omega}|\nabla\phi|^2-C\delta\|q\|^2_{L^2},
		\end{align*}
		due to Lemma \ref{sobolev inequ} and (\ref{phi-1}) such that we know $\|\nabla \phi\|_{L^6} \leq C \|\nabla^2 \phi \|_{L^2} \leq \|q\|_{L^2}$.
		
		Finally, we estimate the nonlinear terms on the right-hand side of (\ref{step1}),
		\begin{align*}
			\int_{\Omega}h'(\tilde{\rho}) f^{0}q&=\int_{\Omega}h'(\tilde{\rho})q\mbox{div}(qu)\leq C\left(\|q\|_{L^3}\|q\|_{L^6}\|\mbox{div}u\|_{L^2}+\|q\|_{L^3}\|\nabla q\|_{L^2}\|u\|_{L^6}\right)\nonumber\\[2mm]
			&\leq C\delta\left(\|\nabla q\|^2_{L^2}+\|\nabla u\|^2_{L^2}\right),
		\end{align*}
		and
		\begin{align*}
			\int_{\Omega}\tilde{\rho} f\cdot u
			&\leq C\left(\|u\|_{L^3}\|\nabla u\|_{L^2}\|u\|_{L^6}+
			\|q\|_{L^3}\|\nabla q\|_{L^2}\|u\|_{L^6}+\|q\|_{L^3}\|\nabla^2 u\|_{L^2}\|u\|_{L^6}\right)\nonumber\\[2mm]
			&\leq C\delta(|u|_{L^3},|q|_{L^3})\left(\|\nabla q\|^2_{L^2}+\|\nabla u\|^2_{L^2}+\|\nabla^2 u\|^2_{L^2}\right).
		\end{align*}
		Putting the above estimates into (\ref{step1}), using Lemma \ref{prop-imp}, we conclude that
		\begin{align*}
			&\frac{1}{2}\frac{d}{dt}\int_{\Omega}\left(\tilde{\rho}|u|^2+h'(\tilde{\rho})|q|^2+|\nabla \phi|^2\right)dx+c\|\nabla u\|^2_{L^2}
			-\int_{\partial\Omega}S(n)u\cdot u+\int_{\partial\Omega} \alpha^{+} |u|^2\nonumber\\[2mm]
			&\leq C^{*}|\alpha^{-}|_{L^{\infty}} R \|\nabla u\|^2_{L^2}+
			C\delta \left( \|q\|_{H^1}^2+\|\nabla u\|_{H^1}^2\right),
		\end{align*}
		which yields the basic energy estimate (\ref{result-1}) by choosing $R$ small enough to satisfy $C^{*}|\alpha^{-}|_{L^{\infty}} R < c$ thus the corresponding term on the right-hand side can be moved to the left-hand side.
	\end{proof}
	
	Next, we prove estimates of the first-order derivatives $\nabla q$  and  $\mbox{div}u$.
	\begin{lem}\label{lem-div}
		Suppose that the conditions in Proposition \ref{prop} hold, the for any $\varepsilon>0$, we have that
		\begin{align}\label{result-2}
			&\frac{1}{2}\frac{d}{dt}\int_{\Omega}\left(\tilde{\rho}|\mbox{div}u|^2+h'(\tilde{\rho})|\nabla q|^2\right)+c\|\nabla \mbox{div}u\|_{L^2}^2\nonumber\\[2mm]
			&\leq
			\varepsilon\left(\|\nabla u_t\|^2_{L^2}+\|q\|_{H^1}^2\right) +C\|\nabla u\|_{L^2}^2+C\delta\mathcal{D}^2(t),
		\end{align}
		where $c$ and $C$ are positive constants with $c$ only depends on $\mu$ and $\lambda$, and constant $C>0$ independent of $t$.
	\end{lem}

	\begin{proof}
		Computing the following integral
		\begin{equation*}
			\int_{\Omega}\nabla(\ref{equation-1})_1\cdot \nabla(h'(\tilde{\rho})q)-(\ref{equation-1})_2\cdot  \tilde{\rho}\nabla\mbox{div}u,
		\end{equation*}
		we have
		\begin{align}\label{step2}
			&(2\mu+\lambda)\|\nabla\mbox{div}u\|_{L^2}^2-\int_{\Omega}\tilde{\rho}u_{t}\cdot \nabla \mbox{div}u
			+\int_{\Omega}\nabla\mbox{div}(\tilde{\rho} u)\cdot \nabla(h'(\tilde{\rho})q)-\int_{\Omega}\tilde{\rho} \nabla(h'(\tilde{\rho})q)\cdot \nabla\mbox{div}u\nonumber\\[2mm]
			&=\mu\int_{\Omega}\mbox{curl}^2u\cdot\nabla\mbox{div}u-\int_{\Omega}\tilde{\rho}\nabla\phi\cdot \nabla \mbox{div}u
			+\int_{\Omega}\nabla f^{0}\cdot\nabla(h'(\tilde{\rho})q)-\int_{\Omega}\tilde{\rho} f\cdot\nabla\mbox{div}u.
		\end{align}

		Integrating by parts with $(u\cdot n)|_{\partial\Omega}=0$, it holds
		\begin{align*}
			-\int_{\Omega}\tilde{\rho}u_{t}\cdot \nabla \mbox{div}u&=\int_{\Omega}\tilde{\rho}\mbox{div}u_{t} \mbox{div}u+\int_{\Omega} u_t\cdot\nabla\tilde{\rho} \mbox{div}u\nonumber\\[2mm]
			&\geq \frac{1}{2}\frac{d}{dt}\int_{\Omega}\tilde{\rho}|\mbox{div}u|^2-\|\nabla\tilde{\rho}\|_{L^3}\|u_t\|_{L^6} \|\mbox{div}u\|_{L^2}\nonumber\\[2mm]
			&\geq \frac{1}{2}\frac{d}{dt}\int_{\Omega}\tilde{\rho}|\mbox{div}u|^2-\varepsilon\|\nabla u_t\|^2_{L^2}-C \|\mbox{div}u\|^2_{L^2},
		\end{align*}
		where we have used the estimate of $\tilde{\rho}$ in Lemma \ref{lem-rho}.
		
		A direct calculation gives
		\begin{align*}
			&\int_{\Omega}\nabla\mbox{div}(\tilde{\rho} u)\cdot \nabla(h'(\tilde{\rho})q)-\int_{\Omega}\tilde{\rho} \nabla(h'(\tilde{\rho})q)\cdot \nabla\mbox{div}u\nonumber\\[2mm]
			&=\int_{\Omega}(u\nabla^2\tilde{\rho} +\nabla\tilde{\rho} \nabla u+\mbox{div}u\nabla\tilde{\rho})\cdot \nabla(h'(\tilde{\rho})q)\nonumber\\[2mm]
			&\geq \varepsilon \|\nabla(h'(\tilde{\rho})q)\|_{L^2}^2-C\|\nabla u\|_{L^2}^2.
		\end{align*}

		Next, we shall estimate the terms on the right-hand side of (\ref{step2}).
		Integrating by parts with $(\mbox{curl}u \times n)|_{\partial\Omega}=((2S(n)-\alpha I)u)_{\tau}$, and applying Proposition \ref{lem-boundary-key} with $v=((2S(n)-\alpha I)u)_{\tau}$ since we have the boundary condition $(u \cdot n)|_{\pl \Omega}=0$, $f=\mbox{div} u$, we can obtain
		\begin{align}
			\int_{\Omega}\mbox{curl}^2u\cdot\nabla\mbox{div}udx&=- \int_{\partial\Omega}(\mbox{curl}u\times n)\cdot\nabla\mbox{div}u \nonumber\\[2mm]
			&=- \int_{\partial\Omega}((2S(n)-\alpha I)u)_{\tau}\cdot\nabla\mbox{div}u\nonumber\\[2mm]
			&\leq C\|\nabla\mbox{div}u\|_{L^2}\|\nabla ((2S(n)-\alpha I)u)\|_{L^2}\nonumber\\[2mm]
			&\leq \epsilon \|\nabla\mbox{div}u\|_{L^2}^2+C\|\nabla u\|_{L^2}^2,
		\end{align}
		for any $\epsilon>0$.

		Integrating by parts and using the boundary condition $(\nabla\phi\cdot n)|_{\partial\Omega}=0$, one has
		\begin{align*}
			-\int_{\Omega}\tilde{\rho}\nabla\phi\cdot \nabla \mbox{div}u
			&=\int_{\Omega}\tilde{\rho}\Delta\phi \mbox{div}u +\int_{\Omega}\nabla \tilde{\rho}\cdot\nabla \phi \mbox{div}u\nonumber\\[2mm]
			&\leq \|\tilde{\rho}\|_{L^{\infty}}\|\Delta\phi\|_{L^2}\|\mbox{div}u\|_{L^2}+\|\nabla \tilde{\rho}\|_{L^3}\|\nabla \phi\|_{L^6}\|\mbox{div}u\|_{L^2}\nonumber\\[2mm]
			&\leq \varepsilon\|q\|_{L^2}^2+C\|\nabla u\|_{L^2}^2,
		\end{align*}
		due to (\ref{phi-1}).

		It remains to control the nonlinear terms on (\ref{step2}). One can easily check that
		\begin{align*}
			\int_{\Omega}\nabla f^{0}\cdot\nabla(h'(\tilde{\rho})q)\leq C\delta\left(\|\nabla q\|^2_{H^1}+\|\nabla^2u\|_{L^2}^2\right),
		\end{align*}
		and
		\begin{align*}
			\int_{\Omega}\tilde{\rho} f\cdot\nabla\mbox{div}u&\leq \epsilon \|\nabla\mbox{div}u\|^2_{L^2}+C\|f\|^2_{L^2}\nonumber\\[2mm]
			&\leq \epsilon \|\nabla\mbox{div}u\|^2_{L^2}+C\left(\|u\|^2_{L^{\infty}}\|\nabla u\|^2_{L^2}
			+\|q\|^2_{L^{\infty}}\|\nabla q\|^2_{L^2}+\|q\|^2_{L^{\infty}}\|\nabla^2 u\|^2_{L^2}\right)\nonumber\\[2mm]
			&\leq  \epsilon \|\nabla\mbox{div}u\|^2_{L^2}+C\delta\|\nabla u\|^2_{H^1}.
		\end{align*}
		Putting the above estimates into (\ref{step2}) and choosing $\epsilon$ small, we conclude (\ref{result-2}).
	\end{proof}

	In order to obtain the energy estimates of $\nabla u$ , in view of Lemma \ref{prop-imp},
	we need to control $\mbox{curl} u$, which is given in the following lemma.
	\begin{lem}\label{lem-w}
		Under the assumptions in Proposition \ref{prop} and for any $\varepsilon>0$, there exist positive constants $c$ depending on $\mu$ and $\lambda$ and $C$ independent of $t$ such that
		\begin{align}\label{result-3}
			&\frac{1}{2}\frac{d}{dt}\int_{\Omega}\tilde{\rho}|\mbox{curl} u|^2+c\|\mbox{curl}^2u\|_{L^2}^2\nonumber\\[2mm]
			&\leq \varepsilon\left(\|\nabla\mbox{div}u\|_{L^2}^2+\|\nabla u_t\|_{L^2}^2+\|q\|_{H^1}^2\right)+C\|\nabla u\|_{L^2}^2+C\delta\mathcal{D}^2(t).
		\end{align}
	\end{lem}

	\begin{proof}
		Multiplying the equation (\ref{equation-u}) by $\tilde{\rho}\mbox{curl}^2u$, and integrating it over $\Omega$, we will get
		\begin{align}\label{step3}
			&\int_{\Omega}\tilde{\rho} u_{t}\cdot\mbox{curl}^2u+\mu\int_{\Omega}|\mbox{curl}^2u|^2\nonumber\\[2mm]
			&=-\int_{\Omega}\tilde{\rho}\nabla (h'(\tilde{\rho})q)\cdot\mbox{curl}^2u+(2\mu+\lambda)\int_{\Omega}(2\mu+\lambda)\nabla \mbox{div}u\cdot\mbox{curl}^2u+\int_{\Omega}\tilde{\rho}\nabla\phi\cdot\mbox{curl}^2u+\int_{\Omega}\tilde{\rho} f\cdot\mbox{curl}^2u.
		\end{align}
		Integrating by parts for both $t$- and $x$-variables,  the first term on the left-hand side becomes
		\begin{align*}
			\int_{\Omega}\tilde{\rho} u_{t}\cdot\mbox{curl}^2u&=\int_{\Omega}\tilde{\rho}\mbox{curl}u_{t}\cdot\mbox{curl}u
			+\int_{\Omega}\nabla\tilde{\rho}\times u_{t}\cdot\mbox{curl}u
			-\int_{\partial\Omega}\tilde{\rho}\mbox{curl}u\times n\cdot u_{t}\nonumber\\[2mm]
			&=\frac{1}{2}\frac{d}{dt}\int_{\Omega}\tilde{\rho}|\mbox{curl}u|^2
			+\int_{\Omega}\nabla\tilde{\rho}\times u_{t}\cdot\mbox{curl}u
			-\int_{\partial\Omega}\tilde{\rho}(2S(n)-\alpha I)u\cdot u_{t}\nonumber\\[2mm]
			&\geq \frac{1}{2}\frac{d}{dt}\int_{\Omega}\tilde{\rho}|\mbox{curl}u|^2
			-\|\nabla\tilde{\rho}\|_{L^3} \|u_{t}\|_{L^6}\|\nabla u\|_{L^2}
			-C\|\nabla u\|\|\nabla u_t\|\nonumber\\[2mm]
			&\geq \frac{1}{2}\frac{d}{dt}\int_{\Omega}\tilde{\rho}|\mbox{curl}u|^2
			-\varepsilon \|\nabla u_{t}\|^2_{L^2}
			-C \|\nabla u\|^2_{L^2},
		\end{align*}
		where we used the boundary condition $(\mbox{curl}u\times n)|_{\partial\Omega}=(2S(n)-\alpha I)u$ and the trace theorem (\ref{est-boundary}) to deal with the boundary term.

		For the first term on the right-hand side of (\ref{step3}), we  integrate by parts with the boundary condition $\mbox{curl}\times n=(2S(n)- \alpha I)u$  again to  get
		\begin{align*}
			-\int_{\Omega}\tilde{\rho}\nabla (h'(\tilde{\rho})q)\cdot\mbox{curl}^2u&=-\int_{\Omega}\nabla\tilde{\rho}\times\nabla (h'(\tilde{\rho})q)\cdot\mbox{curl}u+\int_{\partial\Omega}\tilde{\rho}\nabla (h'(\tilde{\rho})q)\cdot(\mbox{curl}u\times n)\nonumber\\[2mm]
			&=\int_{\Omega}\nabla\tilde{\rho}\times\nabla (h'(\tilde{\rho})q)\cdot\mbox{curl}u+\int_{\partial\Omega}\tilde{\rho}(2S(n)-\alpha I)u\cdot\nabla (h'(\tilde{\rho})q)\nonumber\\[2mm]
			&\leq C\|\nabla\tilde{\rho}\|_{L^{\infty}}\|\nabla u\|_{L^2} \|\nabla (h'(\tilde{\rho})q)\|_{L^2}+C\|\nabla u\|_{L^2} \|\nabla (h'(\tilde{\rho})q)\|_{L^2}\nonumber\\[2mm]
			&\leq \varepsilon \|\nabla (h'(\tilde{\rho})q)\|_{L^2}^2+C \|\nabla u\|_{L^2}^2,
		\end{align*}
		where we used Proposition \ref{lem-boundary-key} for choosing $v=(2S(n)-\alpha I)u$ and $f=h'(\tilde{\rho})q$ to estimate the boundary term.
		Similarly, the third term on the right-hand side of (\ref{step3}) becomes
		\begin{align*}
			(2\mu+\lambda)\int_{\Omega}\nabla \mbox{div}u\cdot\mbox{curl}^2u&= -(2\mu+\lambda)\int_{\pl \Omega}\mbox{curl}u\times n  \cdot \nabla \mbox{div}u\nonumber\\
			&= - \int_{\pl \Omega}(2\mu+\lambda) (2S(n)-\alpha I)u \cdot \nabla \mbox{div}u\nonumber\\
			&\leq \varepsilon \|\nabla \mbox{div}u\|_{L^2}^2+C \|\nabla u\|_{L^2}^2,
		\end{align*}
		where we also applied the Proposition \ref{lem-boundary-key}.

		Integrating by parts and using (\ref{est-boundary}), we deduce that
		\begin{align*}
			\int_{\Omega}\tilde{\rho}\nabla \phi\cdot\mbox{curl}^2u
			&=\int_{\Omega}\nabla\tilde{\rho}\times\nabla \phi\cdot\mbox{curl}u-\int_{\partial\Omega}\tilde{\rho}(2S(n)-\alpha I)u\cdot\nabla \phi\nonumber\\[2mm]
			&\leq \|\nabla\tilde{\rho}\|_{L^3}\|\nabla \phi\|_{L^6}\|\mbox{curl}u\|_{L^2}+C\|\nabla u\|_{L^2}\|\nabla^2\phi\|_{L^2}\nonumber\\[2mm]
			&
			\leq \varepsilon \|q\|_{L^2}^2+C \|\nabla u\|_{L^2}^2,
		\end{align*}
		thanks to (\ref{phi-1}) with (\ref{equation-1}) such that $\Delta \phi = q$ and $(\nabla \phi \cdot n)|_{\pl \Omega}=0$. Finally, the nonlinear term can be controlled by
		\begin{align*}
			\int_{\Omega}\tilde{\rho}f\cdot\mbox{curl}^2u&\leq C\left(\|u\|_{L^{\infty}}\|\nabla u\|_{L^2}
			+\|q\|_{L^{\infty}}\|\nabla q\|_{L^2}+\|q\|_{L^{\infty}}\|\nabla^2 u\|_{L^2}\right) \|\mbox{curl}^2 u\|_{L^2}\nonumber\\[2mm]
			&\leq C\delta(\|u\|_{L^{\infty}},\|q\|_{L^{\infty}})\left(\|\nabla q\|^2_{L^2}+\|\nabla u\|^2_{L^2}+\|\nabla^2 u\|^2_{L^2}\right).
		\end{align*}
		Combining the above estimates and noting the fact $\|\nabla (h'(\tilde{\rho})q)\|^2\leq C\left(\|q\|^2+\|\nabla q\|^2\right)$ by the boundedness of $|h'(\tilde{ \rho})|$ and $|h''(\tilde{ \rho})|$ as shown in (\ref{hb}), we end the proof.
	\end{proof}

	\begin{lem}\label{lem-q}
		Suppose that the conditions in Proposition \ref{prop} hold, there is a positive constant $C$ independent of $t$, such that
		\begin{align}\label{result-4}
			\frac{d}{dt}\int_{\Omega}u\cdot\nabla(h'(\tilde{\rho})q)+\|\sqrt{h'(\tilde{\rho})}q\|^2 +\|\nabla (h'(\tilde{\rho})q)\|^2
			\leq C\left(\|q_t\|^2+\|\nabla u\|_{H^1}^2\right)+C\delta\mathcal{D}^2(t).
		\end{align}
	\end{lem}

	\begin{proof}
		Computing $\int_{\Omega}(\ref{equation-1})_2\cdot \nabla(h'(\tilde{\rho})q)$, one directly has
		\begin{align}\label{step4}
			&\|\nabla(h'(\tilde{\rho})q)\|^2+\int_{\Omega}u_t\cdot \nabla(h'(\tilde{\rho})q)
			-\int_{\Omega}\nabla\phi\cdot\nabla(h'(\tilde{\rho})q)\nonumber\\[2mm]
			&=\int_{\Omega}\tilde{\rho}^{-1}\left(\mu\Delta u+(\mu+\lambda)\nabla\mbox{div}u\right)\cdot \nabla(h'(\tilde{\rho})q)+\int_{\Omega}f\cdot\nabla(h'(\tilde{\rho})q).
		\end{align}
		
		Integrating by parts for both $t$- and $x$-variables with the boundary condition $(u\cdot n)|_{\partial\Omega}=0$,  the second term on the left-hand side can be estimated by
		\begin{align*}
			\int_{\Omega}u_t\cdot \nabla(h'(\tilde{\rho})q)&=\frac{d}{dt}\int_{\Omega}u\cdot \nabla(h'(\tilde{\rho})q)-\int_{\Omega}u\cdot \nabla(h'(\tilde{\rho})q_t)\nonumber\\[2mm]
			&=\frac{d}{dt}\int_{\Omega}u\cdot \nabla(h'(\tilde{\rho})q)+\int_{\Omega}h'(\tilde{\rho})q_t\mbox{div}u\nonumber\\[2mm]
			&\geq \frac{d}{dt}\int_{\Omega}u\cdot \nabla(h'(\tilde{\rho})q)-C(\|q_t\|_{L^2}^2+\|\mbox{div}u\|_{L^2}^2),
		\end{align*}
		where we used the fact that $0<C_1<h'(\tilde{ \rho})<C_2$.
		
		Integrating by parts with the boundary condition $\nabla \phi \cdot n|_{\pl \Omega}=0$ and noting that $\Delta\phi=q$,  the third term on the left hand side of (\ref{step4}) is
		\begin{align*}
			-\int_{\Omega}\nabla\phi\cdot\nabla(h'(\tilde{\rho})q)=\int_{\Omega} h'(\tilde{\rho})q\Delta\phi =\int_{\Omega} h'(\tilde{\rho}) q^2.
		\end{align*}
		Applying Cauchy's inequality directly, the integrals on the right-hand side of (\ref{step4}) can be controlled by
		\begin{align*}
			RHS&\leq\frac{1}{2}\|\nabla(h'(\tilde{\rho})q)\|_{L^2}^2+C\|\nabla^2u\|_{L^2}^2+C\|f\|_{L^2}^2 \nonumber\\[2mm]
			&\leq
			\frac{1}{2}\|\nabla(h'(\tilde{\rho})q)\|_{L^2}^2+C\|\nabla^2u\|_{L^2}^2+
			C\delta\left(\|\nabla q\|_{L^2}^2+\|\nabla u\|^2_{H^1}\right).
		\end{align*}
		Plugging the above estimates into (\ref{step4}), we then conclude (\ref{result-4}).
	\end{proof}

	Next, we prove estimates for the first-order temporal derivatives $(u_t,q_t,\nabla\phi_t)$.
	\begin{lem}\label{lem-ut}
		Under the assumptions in Proposition \ref{prop} and $\alpha>0$, then for the solutions $(q,u,\nabla\phi)$ of (\ref{equation-1}) in any exterior domains $\Omega=\mathbb{R}^3\setminus B_{R}$,
		it holds that
		\begin{align}\label{result5}
			\frac{1}{2}\frac{d}{dt}\int_{\Omega}\left(\tilde{\rho}|u_t|^2+h'(\tilde{\rho})|q_t|^2+|\nabla \phi_t|^2\right)+c\|\nabla u_t\|^2
			\leq C\delta\mathcal{D}^2(t),
		\end{align}
		where $c$ is a positive constant depends on $\mu$ and $\lambda$ and $C>0$ is a constant independent of $t$. If $\alpha^{-}\neq0$, then the estimate (\ref{result5}) is also valid for exterior domains $\Omega=\mathbb{R}^3\setminus B_{R}$ with  small $R$.
	\end{lem}

	\begin{proof}
		Computing the following integral
		\begin{equation*}
			 \int_{\Omega}\partial_{t}(\ref{equation-1})_{1}h'(\tilde{\rho})q_t+\partial_{t}(\ref{equation-1})_2\cdot\tilde{\rho} u_t,
		\end{equation*}
		we get
		\begin{align*}
			&\int_{\Omega} h'(\tilde{ \rho}) q_{tt} q_t + h'(\tilde{ \rho}) \mbox{div}(\tilde{ \rho} u_t) q_t + \tilde{ \rho}  u_{tt} \cdot u_t + \tilde{ \rho}\nabla(h'(\tilde{ \rho})q_t) \cdot u_t - (\mu \Delta u_t + (\mu +\lambda) \nabla \mbox{div} u_t) \cdot u_t - \tilde{ \rho} \nabla \phi_{t} \cdot u_t \nonumber\\[2mm]
			&=\int_{\Omega} \pl_t f^0 h'(\tilde{ \rho}) q_t + \tilde{ \rho} \pl_t f \cdot u_t.\nonumber
		\end{align*}
		Note that differentiation of the system (\ref{equation-1}) with respect to $t$ will keep the boundary conditions that we have $\pl_t u \cdot n|_{\pl \Omega}=0$, $\mbox{curl} u_t \times n|_{\pl \Omega}= (2S(n)- \alpha I)u_t$, and $\nabla \phi_t \cdot n|_{\pl \Omega}=0$. Therefore, (\ref{result5}) can be proved in the same way as in Lemma \ref{lem-basic}, we then omit the proof here.
	\end{proof}

	\subsection{Estimates for the second-order derivatives}
	Firstly, we shall derive the estimates of $\nabla^2u$. By Lemma \ref{lem-div-curl} and Lemma \ref{lem-ell}, we have
	\begin{equation*}
		\| \nabla^2 u \|_{L^2} \leq C(\|\nabla \mbox{div} u\|_{L^2} + \|\mbox{curl}^2 u\|_{L^2} + \|\nabla u \|_{L^2}),
	\end{equation*}
	which is also  proved in \eqref{2u}.
	It thus remains to estimate $\nabla\mbox{div}u$ and $\mbox{curl}^2u$.

	\begin{lem}\label{nabladiv}
		Suppose that the conditions in Proposition \ref{prop} hold, there are positive constants $c$ and $C$ independent of $t$, such that
		\begin{align}\label{result8}
			 &\frac{2\mu+\lambda}{2}\frac{d}{dt}\int_{\Omega}\tilde{\rho}^{-1}|\nabla\mbox{div}u|^2+\frac{d}{dt}\int_{\Omega}\mbox{div}u_{t}\mbox{div}u+c\|\nabla q_t\|_{L^2}^2 \nonumber\\[2mm]
			&\leq C\left(\|q_t\|_{L^2}^2
			+\|\nabla u\|_{H^1}^2+\|\nabla u_t\|_{L^2}^2\right)
			+C\delta\mathcal{D}^2(t).
		\end{align}
	\end{lem}

	\begin{proof}
		Computing the following integral
		\begin{equation*}
			\int_{\Omega}\tilde{\rho}^{-1}\nabla(\ref{equation-1})_1\cdot \nabla(h'(\tilde{\rho})q_t)-\partial_{t}(\ref{equation-1})_2\cdot\nabla\mbox{div}u,
		\end{equation*}
		we have
		\begin{align}\label{step8}
			&\int_{\Omega}\tilde{\rho}^{-1}\nabla q_t \cdot \nabla(h'(\tilde{\rho})q_t)+\int_{\Omega}\tilde{\rho}^{-1}\nabla\mbox{div}(\rho u)\cdot \nabla(h'(\tilde{\rho})q_t)
			-\int_{\Omega}\nabla(h'(\tilde{\rho})q_t)\cdot\nabla\mbox{div}u \nonumber\\[2mm]
			&-\int_{\Omega}u_{tt}\cdot\nabla\mbox{div}u
			+(2\mu+\lambda)\int_{\Omega}\tilde{\rho}^{-1}\nabla\mbox{div}u_t\cdot\nabla\mbox{div}u
			\nonumber\\[2mm] &=\mu\int_{\Omega}\tilde{\rho}^{-1}\mbox{curl}^2u_t\cdot\nabla\mbox{div}u-\int_{\Omega}\nabla\phi_t\cdot\nabla\mbox{div}u+\int_{\Omega}\tilde{\rho}^{-1}\nabla f^0\cdot\nabla(h'(\tilde{\rho})q_t)-\int_{\Omega}f_{t}\cdot\nabla\mbox{div}u.
		\end{align}
		We can easily check that
		\begin{align*}
			\int_{\Omega}\tilde{\rho}^{-1}\nabla\mbox{div}u_t\cdot\nabla\mbox{div}u
			=\frac{1}{2}\frac{d}{dt}\int_{\Omega}\tilde{\rho}^{-1}|\nabla\mbox{div}u|^2.
		\end{align*}
		Now, we estimate other terms in \eqref{step8} one by one. Firstly,
		\begin{align*}
			&\int_{\Omega}\tilde{\rho}^{-1}\nabla q_t \cdot \nabla(h'(\tilde{\rho})q_t)=\int_{\Omega}\tilde{\rho}^{-1}h'(\tilde{\rho})|\nabla q_t|^2 +\int_{\Omega}\tilde{\rho}^{-1}h''(\tilde{\rho})\nabla\tilde{\rho}\cdot\nabla q_t q_t\nonumber\\[2mm]
			&\geq \int_{\Omega}\tilde{\rho}^{-1}h'(\tilde{\rho})|\nabla q_t|^2 -C \|\nabla\tilde{\rho}\|_{L^{\infty}}\|\nabla q_t\|_{L^2}\| q_t\|_{L^2} \nonumber\\[2mm]
			&\geq c\|\nabla q_t\|_{L^2}^2-C\|q_t\|_{L^2}^2,
		\end{align*}
		where we used \eqref{hb} and  the boundedness of  $h''$, $\tilde{\rho}$ and $\nabla \tilde{\rho}$.

		For the second and third integrals on the left-hand side of \eqref{step8}, adding them together we can obtain
		\begin{align*}
			&\int_{\Omega}\tilde{\rho}^{-1}\nabla\mbox{div}(\tilde{\rho} u)\cdot \nabla(h'(\tilde{\rho})q_t)
			-\int_{\Omega}\nabla(h'(\tilde{\rho})q_t)\cdot\nabla\mbox{div}u \nonumber\\[2mm]
			&=\int_{\Omega}\tilde{\rho}^{-1}\left(\mbox{div}u\nabla \tilde{\rho} +\nabla (u\cdot\nabla\rho) \right)\cdot\nabla(h'(\tilde{\rho})q_t)\nonumber\\[2mm]
			&\geq -C\int_{\Omega}|\nabla\tilde{\rho}||\nabla u||\nabla(h'(\tilde{\rho})q_t)|-C\int_{\Omega}|\nabla^2\tilde{\rho}||u||\nabla(h'(\tilde{\rho})q_t)|\nonumber\\[2mm]
			&\geq -C\|\nabla\tilde{\rho}\|_{H^2}\|\nabla u\|_{L^2}\|\nabla(h'(\tilde{\rho})q_t)\|_{L^2}\nonumber\\[2mm]
			&\geq -\epsilon \|\nabla(h'(\tilde{\rho})q_t)\|^2_{L^2}-C\|\nabla u\|^2_{L^2}.
		\end{align*}

		Integrating by parts for both $t$- and $x$-variables with the boundary condition $u_{tt}\cdot n|_{\partial\Omega}=0$, the fourth term on the left-hand side of (\ref{step8}) becomes
		\begin{align*}
			-\int_{\Omega}u_{tt}\cdot\nabla\mbox{div}u=\int_{\Omega}\mbox{div}u_{tt}\mbox{div}u
			=\frac{d}{dt}\int_{\Omega}\mbox{div}u_{t}\mbox{div}u-\int_{\Omega}|\mbox{div}u_{t}|^2.
		\end{align*}

		For the first term on the right-hand side of \eqref{step8},
		integrating by parts implies that
		\begin{align*}
			\int_{\Omega}\tilde{\rho}^{-1}\mbox{curl}^2u_t\cdot\nabla\mbox{div}u
			&=\int_{\Omega}\mbox{curl}(\tilde{\rho}^{-1}\nabla\mbox{div}u)\cdot  \mbox{curl} u_t -\int_{\partial\Omega}\tilde{\rho}^{-1}\mbox{curl}u_t\times n\cdot\nabla\mbox{div}u\nonumber\\[2mm]
			&=\int_{\Omega} \nabla \tilde{ \rho}^{-1} \times \mbox{curl} u_t \cdot \nabla \mbox{div} u -\int_{\partial\Omega}\tilde{\rho}^{-1}\mbox{curl}u_t\times n\cdot\nabla\mbox{div}u\nonumber\\[2mm]
			&\leq C\left(\|\nabla u_t\|_{L^2}^2+ \|\nabla\mbox{div}u\|_{L^2}^2\right),
		\end{align*}
		where we used  Proposition \ref{lem-boundary-key} with choosing $v=(2S(n)-\alpha I)u_t$ and $f=\mbox{div}u$ to deal with the boundary term.

		The  second term on the right-hand side of \eqref{step8} can be controlled by integrating by parts and considering $\Delta \phi_t = q_t$
		\begin{align*}
			\int_{\Omega} \nabla\phi_t\cdot \nabla\mbox{div}u
			&=-\int_{\Omega}\mbox{div}(\nabla\phi_t)\mbox{div}u
			=-\int_{\Omega}q_t\mbox{div}u\nonumber\\[2mm]
			&\leq\|q_t\|_{L^2}^2+\|\mbox{div}u\|_{L^2}^2.
		\end{align*}
		The nonlinear terms in \eqref{step8}  can be controlled as,
		\begin{align*}
			\int_{\Omega}\tilde{\rho}^{-1}\nabla f^0\cdot\nabla(h'(\tilde{\rho})q_t)&\leq  \epsilon \|\nabla(h'(\tilde{\rho})q_t)\|_{L^2}^2+C\|\nabla f^{0}\|_{L^2}^2\nonumber\\[2mm]
			&\leq  \epsilon \|\nabla q_t\|_{L^2}^2+C\|q_t\|_{L^2}^2+C\delta\left(\|\nabla^2 u\|_{L^2}^2+\|\nabla^2 q\|_{L^2}^2\right),
		\end{align*}
		and
		\begin{align*}
			\int_{\Omega}f_{t}\cdot\nabla\mbox{div}u\leq C\delta \left(\|\nabla  u_t\|_{L^2}^2+\|\nabla^2 u\|_{L^2}^2+\|\nabla q_t\|_{L^2}^2+\|\nabla^2 u_t\|_{L^2}^2
			\right).
		\end{align*}
		Combining the above estimates into (\ref{step8}) and  choosing $\epsilon$ small, we get (\ref{result8}).
	\end{proof}

	\begin{lem}
		Under the assumptions in proposition \ref{prop} and for any $\varepsilon>0$, it holds
		\begin{align}\label{result9}
			\frac{1}{2}\frac{d}{dt}\int_{\Omega}|\mbox{curl}^2u|^2+c\|\mbox{curl}^3u\|^2_{L^2}\leq \varepsilon\|\nabla\mbox{curl}^2u\|_{L^2}^2+C\left(\|\nabla u_{t}\|_{L^2}^2+\|\nabla^2u\|_{L^2}^2\right)
			+C\delta\mathcal{D}^2(t),
		\end{align}
		where the positive constant $c$ depends on $\mu$ and $\lambda$ and the positive constant $C$ doesn't depend on $t$.
	\end{lem}

	\begin{proof}
		Recalling $(\ref{equation-1})_2$
		\begin{equation*}
			u_{t}+\nabla (h'(\tilde{\rho})q)-\tilde{\rho}^{-1}\left(\mu\Delta u+(\mu+\lambda)\nabla \mbox{div}u\right)-\nabla\phi=f,
		\end{equation*}
		taking $\mbox{curl}$ of it will give
		\begin{equation}\label{equ-w}
			w_{t}-\nabla \tilde{\rho}^{-1} \times \left(\mu\Delta u+(\mu+\lambda)\nabla \mbox{div}u\right) + \tilde{ \rho}^{-1} \mu \mbox{curl}^2 w= \mbox{curl} f,
		\end{equation}
		where we define $w=\mbox{curl}u$.

		Taking the inner product of the equation (\ref{equ-w}) with $\mbox{curl}^2w$, and integrating with respect to $x$, we obtain
		\begin{align}\label{step9}
			\int_{\Omega}w_{t}\cdot \mbox{curl}^2w+\mu\int_{\Omega}\tilde{\rho}^{-1}|\mbox{curl}^2w|^2
			=\int_{\Omega}\nabla\tilde{\rho}^{-1}\times (\mu\Delta u+(\mu+\lambda)\nabla\mbox{div}u)\cdot \mbox{curl}^2w+\int_{\Omega}\mbox{curl} f\cdot\mbox{curl}^2w.
		\end{align}

		Integrating by parts with the boundary condition $(w_t\times n)|_{\partial\Omega}=(2S(n)-\alpha I)u_t$, we have
		\begin{align*}
			&\int_{\Omega}w_{t}\cdot \mbox{curl}^2w=\int_{\Omega}\mbox{curl}w_{t}\cdot \mbox{curl}w
			+\int_{\partial\Omega}w_{t}\times n\cdot \mbox{curl}w\nonumber\\[2mm]
			&=\frac{1}{2}\frac{d}{dt}\|\mbox{curl}w\|_{L^2}^2+\int_{\partial\Omega}(2S(n)-\alpha I)u_t\cdot \mbox{curl}w\nonumber\\[2mm]
			&\geq \frac{1}{2}\frac{d}{dt}\|\mbox{curl}w\|_{L^2}^2-C\|\nabla u_{t}\|_{L^2} \|\nabla\mbox{curl}w\|_{L^2}\nonumber\\[2mm]
			&\geq \frac{1}{2}\frac{d}{dt}\|\mbox{curl}w\|_{L^2}^2-\varepsilon\|\nabla\mbox{curl}^2u\|_{L^2}^2-C\|\nabla u_{t}\|_{L^2}^2,
		\end{align*}
		where we used the classical Sobolev trace theorem.
		
		Applying the Cauchy's inequality for the terms on the right-hand side directly, we have
		\begin{align*}
			&\int_{\Omega}\nabla\tilde{\rho}^{-1}\times (\mu\Delta u+(\mu+\lambda)\nabla\mbox{div}u)\cdot \mbox{curl}^2w\nonumber\\[2mm]
			&\leq C\|\nabla\tilde{\rho}\|_{L^{\infty}}\|\nabla^2u\|_{L^2}\|\mbox{curl}^2w\|_{L^2}\nonumber\\[2mm]
			&\leq \epsilon \|\mbox{curl}^2w\|^2_{L^2}+C\|\nabla^2u\|^2_{L^2},
		\end{align*}
		and
		\begin{align*}
			&\int_{\Omega}\mbox{curl} f\cdot\mbox{curl}^2w\leq  \epsilon \|\mbox{curl}^2w\|^2_{L^2}+C\|\nabla f\|_{L^2}^2\nonumber\\[2mm]
			&\leq \epsilon \|\mbox{curl}^2w\|^2_{L^2}+C\delta\left(\|\nabla^2 u\|_{L^2}^2+\|\nabla^3 u\|_{L^2}^2+\|\nabla^2 q\|_{L^2}^2\right).
		\end{align*}
		We then finished the proof of (\ref{result9}) after putting the above estimates into (\ref{step9}) and choosing $\epsilon$ small.
	\end{proof}

	Next, we prove estimates of the second-order derivatives $(\nabla u_t, \nabla q_t)$.
	
	\begin{lem}
		Suppose that the conditions in Proposition \ref{prop} hold, there are positive constants $c$ depending on $\mu$ and $\lambda$ and $C$ independent of $t$, such that
		\begin{align}\label{result6}
			&\frac{1}{2}\frac{d}{dt}\int_{\Omega}\left(|\mbox{div}u_t|^2+\tilde{\rho}^{-1}h'(\tilde{\rho})|\nabla q_t|^2\right)
			+\frac{d}{dt}\int_{\Omega}\tilde{\rho}^{-1}h''(\tilde{\rho})\nabla q_{t}\cdot \nabla \tilde{\rho} q_t
			+c\|\nabla \mbox{div}u_t\|_{L^2}^2 \nonumber \\[2mm]
			&\leq C\left(\|q_t\|_{H^1}^2+\|q_{tt}\|_{L^2}^2+\|\nabla u_t\|_{L^2}^2\right)+C\delta\mathcal{D}^2(t).
		\end{align}
	\end{lem}

	\begin{proof}
		We consider the following integral
		\begin{equation*}
			\int_{\Omega}\partial_{t}\nabla(\ref{equation-1})_1\cdot \tilde{\rho}^{-1}\nabla(h'(\tilde{\rho})q_t)-\partial_{t}(\ref{equation-u})\cdot \nabla\mbox{div}u_t
		\end{equation*}
		to get
		\begin{align}\label{step6}
			&-\int_{\Omega} u_{tt}\cdot \nabla\mbox{div}u_t+(2\mu+\lambda)\int_{\Omega}\tilde{\rho}^{-1}|\nabla\mbox{div}u_t|^2+\int_{\Omega}\tilde{\rho}^{-1}\nabla q_{tt}\cdot \nabla(h'(\tilde{\rho})q_t)\nonumber\\[2mm]
			&+\int_{\Omega}\left(\tilde{\rho}^{-1}\nabla\mbox{div}(\tilde{\rho} u_t)\cdot \nabla(h'(\tilde{\rho})q_t)- \nabla(h'(\tilde{\rho})q_t)\cdot\nabla \mbox{div}u_t\right)\nonumber\\[2mm]
			&=\mu\int_{\Omega}\tilde{\rho}^{-1}\mbox{curl}^2u_t\cdot\nabla \mbox{div}u_t
			-\int_{\Omega} \nabla\phi_t\cdot \nabla\mbox{div}u_t+\int_{\Omega}\tilde{\rho}^{-1}\nabla f^0_{t}\cdot \nabla(h'(\tilde{\rho})q_t)- f_t\cdot \nabla\mbox{div}u_t.
		\end{align}

		Integrating by parts with $(u_{tt}\cdot n)|_{\partial\Omega}=0$, one obtains
		\begin{align}\label{step5-1}
			-\int_{\Omega} u_{tt}\cdot \nabla\mbox{div}u_tdx=\int_{\Omega}\mbox{div}u_{tt}\mbox{div}u_t=\frac{1}{2}\frac{d}{dt}\int_{\Omega}|\mbox{div}u_t|^2.
		\end{align}
		
		The third term on the left-hand side of (\ref{step6}) by direct calculation
		\begin{align*}
			&\int_{\Omega}\tilde{\rho}^{-1}\nabla q_{tt}\cdot \nabla(h'(\tilde{\rho})q_t)
			=\int_{\Omega}\tilde{\rho}^{-1}h'(\tilde{\rho})\nabla q_{tt}\cdot \nabla q_t+\int_{\Omega}\tilde{\rho}^{-1}h''(\tilde{\rho})\nabla q_{tt}\cdot \nabla \tilde{\rho} q_t\nonumber\\[2mm]
			&=\frac{1}{2}\frac{d}{dt}\int_{\Omega}\tilde{\rho}^{-1}h'(\tilde{\rho})|\nabla q_{t}|^2
			+\frac{d}{dt}\int_{\Omega}\tilde{\rho}^{-1}h''(\tilde{\rho})\nabla q_{t}\cdot \nabla \tilde{\rho} q_t-\int_{\Omega}\tilde{\rho}^{-1}h''(\tilde{\rho})\nabla q_{t}\cdot \nabla \tilde{\rho} q_{tt}\nonumber\\[2mm]
			&\geq \frac{1}{2}\frac{d}{dt}\int_{\Omega}\tilde{\rho}^{-1}h'(\tilde{\rho})|\nabla q_{t}|^2
			+\frac{d}{dt}\int_{\Omega}\tilde{\rho}^{-1}h''(\tilde{\rho})\nabla q_{t}\cdot \nabla \tilde{\rho} q_t
			-C\left(\|\nabla q_{t}\|_{L^2}^2+\|q_{tt}\|_{L^2}^2\right),
		\end{align*}
		where we used the boundedness of $|h''(\tilde{\rho})|$  and $|\nabla\tilde{\rho}|$.
		The fourth term on the left hand side of (\ref{step6}) can be estimate as follows,
		\begin{align}
			&\int_{\Omega}\left(\tilde{\rho}^{-1}\nabla\mbox{div}(\tilde{\rho} u_t)\cdot \nabla(h'(\tilde{\rho})q_t)- \nabla(h'(\tilde{\rho})q_t)\cdot\nabla \mbox{div}u_t\right)\nonumber\\[2mm]
			&=\int_{\Omega}\tilde{\rho}^{-1}\left(\mbox{div}u_t\nabla \tilde{\rho} +\nabla (u_t\cdot\nabla\tilde{\rho}) \right)\cdot\nabla(h'(\tilde{\rho})q_t)\nonumber\\[2mm]
			&\geq -C\int_{\Omega}|\nabla\tilde{\rho}||\nabla u_t||\nabla(h'(\tilde{\rho})q_t)|-C\int_{\Omega}|\nabla^2\tilde{\rho}||u_t||\nabla(h'(\tilde{\rho})q_t)|\nonumber\\[2mm]
			&\geq -C\left(\|\nabla u_t\|^2_{L^2}+\|\nabla(h'(\tilde{\rho})q_t)\|^2_{L^2}\right)\nonumber\\[2mm]
			&\geq -C\left(\|\nabla u_t\|^2_{L^2}+\|q_t\|^2_{H^1}\right).
		\end{align}
		After integrating by parts with the boundary condition $(\mbox{curl}u_t \times n)|_{\partial\Omega}=(S(n)-\alpha I)u_t$, the first term on the right-hand side of (\ref{step6}) becomes
		\begin{align}
			&\int_{\Omega}\tilde{\rho}^{-1}\mbox{curl}^2u_t\cdot\nabla \mbox{div}u_t\nonumber\\[2mm]
			&=\int_{\Omega}\mbox{curl}(\tilde{\rho}^{-1}\nabla \mbox{div}u_t)\cdot \mbox{curl}u_t
			-\int_{\partial\Omega}\tilde{\rho}^{-1}\nabla \mbox{div}u_t\cdot (\mbox{curl}u_t\times n)\nonumber\\[2mm]
			&=\int_{\Omega}\nabla\tilde{\rho}^{-1}\times \nabla\mbox{div}u_t \cdot \mbox{curl}u_t
			-\int_{\partial\Omega}\tilde{\rho}^{-1}(S(n)-\alpha I)u_t\cdot\nabla \mbox{div}u_t\nonumber\\[2mm]
			&\leq \epsilon\|\nabla\mbox{div}u_t\|_{L^2}^2+C\|\nabla u_t\|_{L^2}^2,
		\end{align}
		where we used Proposition  \ref{lem-boundary-key} by choosing $v=(S(n)-\alpha I)u_t$ and $f=\mbox{div}u_t$ to deal with the boundary integration.
		For the second term on the right-hand side of \eqref{step6}, integrating by parts with the boundary condition $(\nabla \phi_t \cdot n )|_{\pl \Omega}=0$ and using  $\Delta \phi_{t}=q_t$, one has
		\begin{align}\label{step5-2}
			\int_{\Omega} \nabla\phi_t\cdot \nabla\mbox{div}u_t
			&=-\int_{\Omega}\mbox{div}(\nabla\phi_t)\mbox{div}u_t
			=-\int_{\Omega}q_t\mbox{div}u_t \leq\|q_t\|^2_{L^2}+\|\mbox{div}u_t\|^2_{L^2}.
		\end{align}
		
		Noted that the integral in the nonlinear term with $f^0_t$ on the right-hand side of (\ref{step6}) contains the term with $\nabla^2q_t$, which is not included in $\mathcal{D}(t)$, thus we shall work on it carefully.
		\begin{align}\label{nonlinear-1}
			&\int_{\Omega}\tilde{\rho}^{-1}\nabla f^0_{t}\cdot \nabla(h'(\tilde{\rho})q_t)=\int_{\Omega}\tilde{\rho}^{-1}\partial_{t}\nabla\mbox{div}(qu)\cdot \nabla(h'(\tilde{\rho})q_t)\nonumber\\[2mm]
			&=\int_{\Omega}\tilde{\rho}^{-1} u\cdot\nabla^2q_t\cdot\nabla (h'(\tilde{\rho})q_t)
			+\int_{\Omega}\tilde{\rho}^{-1}\Big(\mbox{div}u\nabla q_t+q_t\nabla\mbox{div}u\nonumber\\[2mm]
			&+\mbox{div}u_t\nabla q+q \nabla\mbox{div}u_t+\nabla u_t\cdot \nabla q
			+u_t\cdot\nabla^2q+\nabla u\cdot \nabla q_t\Big)\cdot\nabla (h'(\tilde{\rho})q_t),
		\end{align}
		excepting the first term on the right-hand side of \eqref{nonlinear-1}, one can check that other terms can be easily controlled by $C \delta \mathcal{D}^2(t)$.
		For the  first term on the right-hand side of \eqref{nonlinear-1},
		integrating by parts with the boundary condition $(u\cdot n)|_{\partial\Omega}=0$,  we obtain
		\begin{align*}
			&\int_{\Omega}\tilde{\rho}^{-1}u\cdot\nabla^2q_t \cdot \nabla (h'(\tilde{\rho}) q_t)\\[2mm]
			& = \int_{\Omega}\tilde{\rho}^{-1}h'(\tilde{\rho})u\cdot\nabla^2q_t \cdot \nabla q_t+
			\int_{\Omega}\tilde{\rho}^{-1}u\cdot\nabla^2q_t \cdot \nabla h'(\tilde{\rho}) q_t\\[2mm]
			& =-\frac{1}{2}\int_{\Omega}\mbox{div}\left(\tilde{\rho}^{-1}h'(\tilde{\rho})u\right)|\nabla q_t|^2
			-\int_{\Omega}\pl_{i}\left(\tilde{\rho}^{-1}u_i\pl_j h'(\tilde{\rho}) q_t\right)\pl_{j} q_t
			\\[2mm]
			&\leq \delta \|\nabla q_t\|_{L^2}^2.
		\end{align*}
		Noting $\|\nabla q_t\|_{L^2}^2$ is included  in $\mathcal{D}(t)$, so
		\begin{align}\label{step5-3}
			\int_{\Omega}\tilde{\rho}^{-1}\nabla f^0_{t}\cdot \nabla(h'(\tilde{\rho})q_t)\leq C\delta\mathcal{D}^2(t).
		\end{align}
		For the other nonlinear  term in \eqref{step6}, we apply the Cauchy's inequality directly to get
		\begin{align}\label{step5-4}
			\int_{\Omega} f_t\cdot \nabla\mbox{div}u_t&\leq \epsilon\|\nabla\mbox{div}u_t\|_{L^2}^2+ C\|f_t\|_{L^2}^2 \leq\epsilon\|\nabla\mbox{div}u_t\|_{L^2}^2+ C\delta\mathcal{D}^2(t).
		\end{align}
		Finally, putting (\ref{step5-1})-(\ref{step5-2}), (\ref{step5-3})  and (\ref{step5-4}) into (\ref{step6}) and choosing $\epsilon$ small, we conclude the desired estimate (\ref{result6}).
	\end{proof}

	\begin{lem}
		Under the assumptions in Proposition \ref{prop} and for any $\varepsilon>0$, it holds that
		\begin{align}\label{result7}
			&\frac{1}{2}\frac{d}{dt}\int_{\Omega}|\mbox{curl} u_t|^2 - \frac{1}{2}\frac{d}{dt} \int_{\pl \Omega}u_t(S(n)-\alpha I)\cdot u_t
			+c\|\mbox{curl}^2u_t\|_{L^2}^2\nonumber\\[2mm]
			&\leq \varepsilon(\|\nabla^2u_t\|_{L^2}^2+\|q_t\|^2_{H^1})+ C\|\nabla u_t\|_{L^2}^2 +C\delta\mathcal{D}^2(t),
		\end{align}
		where $c>0$ is a constant depending on $\mu$ and $\lambda$ and positive constant $C$ independent of $t$.
	\end{lem}

	\begin{proof}
		Computing the following integral
		\begin{equation*}
			\int_{\Omega}
			\partial_t(\ref{equation-u})\cdot \mbox{curl}^2u_t,
		\end{equation*}
		we get
		\begin{align}\label{step7}
			&\int_{\Omega} u_{tt}\cdot\mbox{curl}^2u_t+\mu\int_{\Omega}\tilde{\rho}^{-1}|\mbox{curl}^2u_t|^2
			=-\int_{\Omega}\nabla (h'(\tilde{\rho})q_t)\cdot\mbox{curl}^2u_t\nonumber\\[2mm]
			&+(2\mu+\lambda)\int_{\Omega}\tilde{\rho}^{-1}\nabla \mbox{div}u_t\cdot\mbox{curl}^2u_t+\int_{\Omega}\nabla\phi_t\cdot\mbox{curl}^2u_t+\int_{\Omega} f_t\cdot\mbox{curl}^2u_t.
		\end{align}
		Integrating by parts for both $t$- and $x$-variables,  and using  the boundary condition $(\mbox{curl}u_{t}\times n)|_{\partial\Omega}=(S(n)-\alpha I)u_{t}$, the first term on the left-hand side of \eqref{step7} becomes
		\begin{align*}
			\int_{\Omega} u_{tt}\cdot\mbox{curl}^2u_t&=\int_{\Omega}\mbox{curl}u_{tt}\cdot\mbox{curl}u_t
			-\int_{\partial\Omega}\mbox{curl}u_t\times n\cdot u_{tt}\nonumber\\[2mm]
			&=\frac{1}{2}\frac{d}{dt}\int_{\Omega}|\mbox{curl}u_t|^2
			-\int_{\partial\Omega}u_t (S(n)-\alpha I) \cdot u_{tt} \nonumber\\[2mm]
			&= \frac{1}{2}\frac{d}{dt}\int_{\Omega}|\mbox{curl}u_t|^2
			-\frac{1}{2}\frac{d}{dt} \int_{\pl \Omega}u_t (S(n)-\alpha I) \cdot u_t.
		\end{align*}
		For the first term on the right-hand side of \eqref{step7}, we integrate by parts with the boundary condition $(\mbox{curl}u_{t}\times n)|_{\partial\Omega}=(S(n)-\alpha I)u_{t}$ to deduce that,
		\begin{align*}
			-\int_{\Omega}\nabla (h'(\tilde{\rho})q_t)\cdot\mbox{curl}^2u_t&=\int_{\partial\Omega}\nabla (h'(\tilde{\rho})q_t)\cdot(\mbox{curl}u_t\times n)\nonumber\\[2mm]
			&=\int_{\partial\Omega}u_t(S(n)-\alpha I)\cdot\nabla (h'(\tilde{\rho})q_t) \nonumber\\[2mm]
			&\leq C\|\nabla u_t\|_{L^2} \|\nabla (h'(\tilde{\rho})q_t)\|_{L^2} \leq \varepsilon \|\nabla (h'(\tilde{\rho})q_t)\|_{L^2}^2+C \|\nabla u_t\|_{L^2}^2,
		\end{align*}
		where we used Proposition \ref{lem-boundary-key} with choosing $v=(S(n)-\alpha I)u_t$ and $f=h'(\tilde{\rho})q_t$ to estimate the boundary term. Similarly, \begin{align*}
			\int_{\Omega}\tilde{\rho}^{-1}\nabla \mbox{div}u_t\cdot\mbox{curl}^2u_t
			&=\int_{\Omega}\nabla\tilde{\rho}^{-1}\times\nabla \mbox{div}u_t\cdot\mbox{curl}u_t
			-\int_{\partial\Omega}\tilde{\rho}^{-1}\nabla \mbox{div}u_t\cdot(\mbox{curl}u_t\times n)\nonumber\\[2mm]
			&=\int_{\Omega}\nabla\tilde{\rho}^{-1}\times\nabla \mbox{div}u_t\cdot\mbox{curl}u_t
			-\int_{\partial\Omega}\tilde{\rho}^{-1}u_t(S(n)-\alpha I)\cdot\nabla \mbox{div}u_t\nonumber\\[2mm]
			&\leq \varepsilon \|\nabla\mbox{div}u_t\|_{L^2}^2+C\|\nabla u_t\|_{L^2}^2,
		\end{align*}
		where we used Proposition  \ref{lem-boundary-key} again.

		Next, after integrating by parts and using Lemma \ref{lem-ell} to estimate the boundary integration, we have
		\begin{align*}
			\int_{\Omega}\nabla \phi_t\cdot\mbox{curl}^2u_t&=-\int_{\partial\Omega}\nabla \phi_t\cdot(\mbox{curl}u_t\times n)\nonumber\\[2mm]
			&=-\int_{\partial\Omega}u_t(2S(n)-\alpha I)\cdot\nabla \phi_t\nonumber\\[2mm]
			&\leq C\|\nabla u_t\|_{L^2}\|\nabla^2\phi_t\|_{L^2}\nonumber\\[2mm]
			&\leq \varepsilon \|q_t\|_{L^2}^2+C \|\nabla u_t\|_{L^2}^2,
		\end{align*}
		where we used $\|\nabla^2\phi_t\|_{L^2}\leq C\|q_t\|_{L^2}$ as in Lemma \ref{lem-neu}.

		Lastly, the nonlinear term in \eqref{step7} can be controlled by
		\begin{align*}
			\int_{\Omega}f_t\cdot\mbox{curl}^2u_t&\leq \epsilon \|\mbox{curl}^2u_t\|_{L^2}^2+C\|f_t\|_{L^2}^2 \nonumber\\[2mm]
			&\leq \epsilon \|\mbox{curl}^2u_t\|_{L^2}^2+C\delta\left(\|\nabla u_t\|_{L^2}^2+\|\nabla^2 u_t\|_{L^2}^2+\|\nabla q_t\|_{L^2}^2\right).
		\end{align*}
		Combining the above estimates into (\ref{step7}), we deduce the proof of (\ref{result7}).
	\end{proof}

	In order to close the assumption \eqref{small}, we shall provide some estimates for $\nabla^2q$.

	\begin{lem}\label{result2q}
		Suppose that the conditions in Proposition \ref{prop} hold and for any $\varepsilon>0$, we have
		\begin{align*}
			&\frac{d}{dt}\int_{\Omega}|\nabla^2 q|^2+\frac{d}{dt}\int_{\Omega}\tilde{\rho}^{-1}q\nabla^2q\cdot\left(h''(\tilde{\rho})\nabla^2\tilde{\rho}
			 +h'''(\tilde{\rho})\nabla\tilde{\rho}\nabla\tilde{\rho}\right)+2\frac{d}{dt}\int_{\Omega}\tilde{\rho}^{-1}h''\nabla\tilde{\rho}\cdot\nabla^2q\cdot\nabla q
			+c\|\nabla^2\mbox{div}u\|_{L^2}^2\nonumber\\[2mm]
			&\leq \varepsilon\|\nabla^2 q\|^2_{L^2}
			+ C\left( \|\nabla u_t\|_{L^2}^2+\| q_t\|_{H^1}^2+\|\nabla(\mbox{curl}^2u)\|_{L^2}^2+\|\nabla u\|_{H^1}^2+ \|\nabla q\|^2_{H^1} +\|\nabla^2 \phi\|_{L^2}^2\right)+C\delta\mathcal{D}^2(t),
		\end{align*}
		where $c$ and $C$  are positive constants independent of $t$.
	\end{lem}

	\begin{proof}
		Computing the following integral
$$\int_{\Omega}\tilde{\rho}^{-1}\partial_{ij}(\ref{equation-1})_1 \cdot\partial_{ij}(
		h'(\tilde{\rho})q)-\nabla_{i}(\ref{equation-1})_2^{j}\cdot\nabla_{ij}\mbox{div}u,$$
		we have
		\begin{align}\label{2q}
			&(2\mu+\lambda)\int_{\Omega}\tilde{\rho}^{-1}|\nabla^{2}\mbox{div}u|^2
			+\int_{\Omega}\tilde{\rho}^{-1}\partial_{ij}q_t\cdot\partial_{ij}(h'(\tilde{\rho})q)\nonumber\\[2mm]
			&+\int_{\Omega}\tilde{\rho}^{-1}\partial_{ij}\mbox{div}(\tilde{\rho} u)\cdot\partial_{ij}(h'(\tilde{\rho})q)
			-\int_{\Omega}\partial_{ij}(h'(\tilde{\rho})q)\cdot\partial_{ij}\mbox{div}u\nonumber\\[2mm]
			&=\int_{\Omega}\nabla u_t\cdot\nabla^2\mbox{div}u+\mu\int_{\Omega}\tilde{\rho}^{-1}\nabla(\mbox{curl}^2u)\cdot\nabla^2\mbox{div}u\nonumber\\[2mm]
			 &-\int_{\Omega}\nabla\tilde{\rho}^{-1}\cdot(-\mu\mbox{curl}^2u+(2\mu+\lambda)\nabla\mbox{div}u)\cdot\nabla^2\mbox{div}u
			-\int_{\Omega}\nabla^2 \phi\cdot\nabla^2\mbox{div}u\nonumber\\[2mm]
			&+\int_{\Omega}\tilde{\rho}^{-1}\nabla^2f^{0}\cdot \nabla^{2}(h'(\tilde{\rho})q)-\int_{\Omega}\nabla f\cdot\nabla^2\mbox{div}u.
		\end{align}

		Integrating by part for $t$-variable, we deduce through direct calculation
		\begin{align*}
			&\int_{\Omega}\tilde{\rho}^{-1}\partial_{ij}q_t\cdot\partial_{ij}(h'(\tilde{\rho})q)\nonumber\\[2mm]
			&=\int_{\Omega}\tilde{\rho}^{-1}h'(\tilde{\rho})\partial_{ij}q_t\cdot\partial_{ij}q
			 +2\int_{\Omega}\tilde{\rho}^{-1}h''(\tilde{\rho})\partial_{ij}q_t\cdot\partial_{i}q\partial_{j}\tilde{\rho}
			+\int_{\Omega}\tilde{\rho}^{-1}q\partial_{ij}q_t\cdot\left(h''(\tilde{\rho})\partial_{ij}\tilde{\rho}
			+h'''(\tilde{\rho})\partial_i\tilde{\rho}\partial_j\tilde{\rho}\right)\nonumber\\[2mm]
			&=\frac{1}{2}\frac{d}{dt}\int_{\Omega}\tilde{\rho}^{-1}h'(\tilde{\rho})|\nabla^2q|^2
			+2\frac{d}{dt}\int_{\Omega}\tilde{\rho}^{-1}h''\nabla\tilde{\rho}\cdot\nabla^2q\cdot\nabla q
			+\frac{d}{dt}\int_{\Omega}\tilde{\rho}^{-1}q\nabla^2q\cdot\left(h''(\tilde{\rho})\nabla^2\tilde{\rho}
			+h'''(\tilde{\rho})\nabla\tilde{\rho}\nabla\tilde{\rho}\right)\nonumber\\[2mm]
			&-\int_{\Omega}\tilde{\rho}^{-1}h''\nabla\tilde{\rho}\cdot\nabla^2q\cdot\nabla q_t
			-\int_{\Omega}\tilde{\rho}^{-1}q_t\cdot\nabla^2 q\cdot\left(h''(\tilde{\rho})\nabla^2\tilde{\rho}
			+h'''(\tilde{\rho})\nabla\tilde{\rho}\nabla\tilde{\rho}\right)\nonumber\\[2mm]
			&\geq \frac{1}{2}\frac{d}{dt}\int_{\Omega}\tilde{\rho}^{-1}h'(\tilde{\rho})|\nabla^2q|^2
			+2\frac{d}{dt}\int_{\Omega}\tilde{\rho}^{-1}h''\nabla\tilde{\rho}\cdot\nabla^2q\cdot\nabla q
			+\frac{d}{dt}\int_{\Omega}\tilde{\rho}^{-1}q\nabla^2q\cdot\left(h''(\tilde{\rho})\nabla^2\tilde{\rho}
			+h'''(\tilde{\rho})\nabla\tilde{\rho}\nabla\tilde{\rho}\right)\nonumber\\[2mm]
			&-\varepsilon\|\nabla^2 q\|_{L^2}^2-C(\|\nabla q_t\|_{L^2}^2+
			\|q_t\|_{L^2}^2).
		\end{align*}

		Combining the next two terms, we get the following through direct calculation
		\begin{align*}
			&\int_{\Omega}\tilde{\rho}^{-1}\partial_{ij}\mbox{div}(\tilde{\rho} u)\cdot\partial_{ij}(h'(\tilde{\rho})q)
			-\int_{\Omega}\partial_{ij}(h'(\tilde{\rho})q)\cdot\partial_{ij}\mbox{div}u\nonumber\\[2mm]
			&=\int_{\Omega}\tilde{\rho}^{-1}\left(\nabla^2\tilde{\rho} \mbox{div}u+2\nabla\tilde{\rho}\cdot\nabla\mbox{div}u+\nabla^2u\cdot\nabla\tilde{\rho}
			+u\cdot\nabla^{3}\tilde{\rho}+2\nabla u\cdot\nabla^2\tilde{\rho}\right)\cdot \nabla^{2}\mbox{div}u\nonumber\\[2mm]
			&\geq \epsilon \|\nabla^{2}\mbox{div}u\|^2-C(\|\nabla\tilde{\rho}\|^2_{L^{\infty}}+\|\nabla^2\tilde{\rho}\|^2_{L^{3}}+\|\nabla^3\tilde{\rho}\|^2_{L^{3}})
			\left(\|u\|_{L^6}^2+\|\nabla u\|_{L^6}^2+\|\nabla^2u\|_{L^2}^2\right)\nonumber\\[2mm]
			&\geq \epsilon \|\nabla^{2}\mbox{div}u\|^2-C\|\nabla u\|_{H^1}^2,
		\end{align*}
		where we have used the estimate for $\tilde{\rho}$ \eqref{estimate-rho} in Lemma \ref{lem-rho}.
		
		The first four terms on the right-hand side of (\ref{2q}) can be estimated by using Cauchy's inequalities directly
		\begin{align*}
			&\int_{\Omega}\nabla u_t\cdot\nabla^2\mbox{div}u+\mu\int_{\Omega}\tilde{\rho}^{-1}\nabla(\mbox{curl}^2u)\cdot\nabla^2\mbox{div}u\nonumber\\[2mm]
			 &-\int_{\Omega}\nabla\tilde{\rho}^{-1}\cdot(-\mu\mbox{curl}^2u+(2\mu+\lambda)\nabla\mbox{div}u)\cdot\nabla^2\mbox{div}u
			-\int_{\Omega}\nabla^2 \phi\cdot\nabla^2\mbox{div}u\nonumber\\[2mm]
			&\leq \epsilon \|\nabla^2\mbox{div}u\|_{L^2}^2+C\left( \|\nabla u_t\|_{L^2}^2+\|\nabla(\mbox{curl}^2u)\|_{L^2}^2+\|\nabla^2u\|_{L^2}^2+\|\nabla^2 \phi\|_{L^2}^2\right).
		\end{align*}

		The nonlinear terms on the right-hand side of (\ref{2q}) can be estimated as
		\begin{align}\label{2q-1}
			&\int_{\Omega}\tilde{\rho}^{-1}\nabla^2f^{0}\cdot \nabla^{2}(h'(\tilde{\rho})q)=\int_{\Omega}\tilde{\rho}^{-1}\nabla^2\mbox{div}(qu)\cdot \nabla^{2}(h'(\tilde{\rho})q)\nonumber\\[2mm]
			&= \int_{\Omega}\tilde{\rho}^{-1}u\cdot\nabla^3q\cdot \nabla^{2}(h'(\tilde{\rho})q)\nonumber\\[2mm]
			&+\int_{\Omega}\tilde{\rho}^{-1}\left(\nabla^2q \mbox{div}u+2\nabla q\cdot\nabla\mbox{div}u+q\nabla^2\mbox{div}u
			+\nabla^2u\cdot\nabla q+2\nabla u\cdot\nabla^2q\right) \cdot \nabla^{2}(h'(\tilde{\rho})q),
		\end{align}
		except for the first term on the right-hand side, the other terms can be easily controlled by $C\delta\left(\|\nabla q\|_{H^1}^2+\|\nabla^3 u\|_{L^2}^2\right)$.

		For the first term on the right-hand side of (\ref{2q-1}), by using integration by parts with the boundary condition $(u\cdot n)|_{\partial\Omega}=0$, we obtain
		\begin{align*}
			&\int_{\Omega}\tilde{\rho}^{-1}u\cdot\nabla^{2}(h'(\tilde{\rho})q)\cdot\nabla^3q\nonumber\\[2mm]
			&=-\frac{1}{2}\int_{\Omega}\mbox{div}\Big(\tilde{ \rho}^{-1}h'(\tilde{ \rho})u\Big) |\nabla^2q|^2
			-\int_{\Omega}\mbox{div}\Big(\tilde{\rho}^{-1}u\cdot(h''\nabla\tilde{\rho}\cdot\nabla q+h''\nabla^2\tilde{\rho} q+h'''\nabla\tilde{\rho } \cdot\nabla\tilde{\rho} q)\Big)  \cdot\nabla^2q \nonumber\\[2mm]
			&\leq C\delta \left(\|\nabla^2q\|_{L^2}^2+\|\nabla q\|_{L^2}^2\right).
		\end{align*}
		Similarly, we have
		\begin{align*}
			-\int_{\Omega} \nabla f \cdot \nabla^2 \mbox{div} u &\leq \epsilon \|\nabla^2 \mbox{div} u\|^2_{L^2}
+ C\| \nabla f \|^2_{L^2} \nonumber\\[2mm]
			&\leq \epsilon \|\nabla^2 \mbox{div} u\|^2_{L^2} + C \delta \mathcal{D}^2(t).
		\end{align*}
		Putting the above estimates into (\ref{2q}), we end the proof of Lemma \ref{result2q}.
	\end{proof}
	
	At last, we need to estimate $\int_0^t\|\nabla^2q\|^2_{L^2}ds$.
	\begin{lem}\label{lem-2q}
		Suppose that the conditions in Proposition \ref{prop} hold, there is a positive constant $C$ independent of $t$, such that
		\begin{align}\label{result11}
			\|\nabla^2q\|^2_{L^2}\leq C\left(\|q\|_{H^1}^2+\|\nabla u_t\|_{L^2}^2+\|\nabla^2u\|^2_{H^1}+\|\nabla^2\phi\|_{L^2}^2\right)+C\delta \left(\|\nabla^2u\|^2_{H^1}+\|\nabla q\|^2_{H^1}\right).
		\end{align}
	\end{lem}
	\begin{proof}
		Taking $\int_{\Omega}\nabla(\ref{equation-1})_2\cdot \nabla^2(h'(\tilde{\rho})q)$, we have
		\begin{align}\label{step11}
			\|\nabla^2(h'(\tilde{\rho})q)\|^2&=-\int_{\Omega}\nabla u_t\cdot \nabla^2(h'(\tilde{\rho})q)
			+\int_{\Omega}\nabla\Big(\tilde{\rho}^{-1}\left(\mu\Delta u+(\mu+\lambda)\nabla\mbox{div}u\right)\Big)\cdot \nabla^2(h'(\tilde{\rho})q)\nonumber\\[2mm]
			&+\int_{\Omega}\nabla^2\phi\cdot\nabla^2(h'(\tilde{\rho})q)-\int_{\Omega}\nabla f\cdot\nabla^2(h'(\tilde{\rho})q)\nonumber\\[2mm]
			&\equiv K_1-\int_{\Omega}\nabla f\cdot\nabla^2(h'(\tilde{\rho})q).
		\end{align}
		By using the Cauchy's inequality, one obtains
		\begin{align}\label{step11-1}
			K_1\leq \frac{1}{4}\|\nabla^2(h'(\tilde{\rho})q)\|_{L^2}^2+C\left(\|\nabla u_t\|_{L^2}^2+\|\nabla^2u\|^2_{H^1}+\|\nabla^2\phi\|_{L^2}^2\right),
		\end{align}
		and
		\begin{align}\label{step11-2}
			\int_{\Omega}\nabla f\cdot\nabla^2(h'(\tilde{\rho})q)\leq \frac{1}{4}\|\nabla^2(h'(\tilde{\rho})q)\|_{L^2}^2+  C\delta \left(\|\nabla^2u\|^2_{H^1}+\|\nabla q\|^2_{H^1}\right)
		\end{align}
		Putting \eqref{step11-1} and \eqref{step11-2} into \eqref{step11}, and using the facts  $\|\nabla^2q\|_{L^2}\leq C\left(\|\nabla^2(h'(\tilde{\rho})q\|_{L^2}+\|q\|_{H^1}\right)$,
		we obtain the desired estimate \eqref{result11}.
	\end{proof}

	We now combine Lemmas \ref{lem-basic}-\ref{lem-2q} to derive the a priori estimates for the solutions to (\ref{equation-1}).

	{\bf {Proof of  Proposition \ref{prop} }}Firstly, rewriting the equation $\eqref{equation-1}_2$ as
	\begin{equation*}
		-\mu\Delta u(\mu+\lambda)-\nabla \mbox{div}u=\tilde{\rho}\left(-u_{t}-\nabla (h'(\tilde{\rho})q)+
		\nabla\phi+f\right),
	\end{equation*}
	then applying the elliptic estimate (i.e. Lemma \ref{lem-elliptic-1}), we have
	\begin{align*}
		\|\nabla^3u\|_{L^2}\leq C(|\tilde{\rho}|_{L^\infty},|\nabla\tilde{\rho}|_{L^{\infty}})\left(\|u_t\|_{H^1}+\|\nabla (h'(\tilde{\rho}) q)\|_{H^1}+\|\nabla\phi\|_{H^1}+\|f\|_{H^1}\right),
	\end{align*}
	with all the terms on the right-hand side already having estimates in the above Lemmas.
	
	Using the first equation of (\ref{equation-1}), it is easy to see that
	\begin{align*}
		\|q_t\|_{L^2}&\leq \tilde{\rho}\|\nabla u\|_{L^2}+\|\nabla\tilde{\rho}\|_{L^{3}}\|u\|_{L^6}+\|q\|_{L^{\infty}}\|\nabla u\|_{L^2}+\|\nabla q\|_{L^3}\|u\|_{L^6}\nonumber\\[2mm]
		&\leq C\|\nabla u\|_{L^2}+C\delta\|\nabla u\|_{L^2}.
	\end{align*}
	Similarly,
	\begin{align*}
		\|q_{tt}\|_{L^2}&\leq \tilde{\rho}\|\nabla u_t\|_{L^2}+\|\nabla\tilde{\rho}\|_{L^{3}}\|u_t\|_{L^6}+\|q\|_{L^{\infty}}\|\nabla u_t\|_{L^2}\nonumber\\[2mm]
&+\|\nabla q\|_{L^3}\|u_t\|_{L^6}+\|\mbox{div}u\|_{L^3}\| q_t\|_{L^6}+\|u\|_{L^{\infty}}\|\nabla q_t\|_{L^2}\nonumber\\[2mm]
		&\leq C\|\nabla u_t\|_{L^2}+C\delta\left(\|\nabla u_t\|_{L^2}+\|\nabla q_t\|_{L^2}\right).
	\end{align*}
	Therefore, the estimates of $\|q_t\|_{L^2_t({L^2})}$ and $\|q_{tt}\|_{L^2_t({L^2})}$ are obtained.
	
	Next, we show estimates of $\|\nabla^2u\|_{L^2_t({L^2})}$ and $\|\nabla^3u\|_{L^2_t({L^2})}$.
	Recalling
	\begin{align*}
		\|\nabla^2 u\|_{L^2}\leq C\left(\|\nabla\mbox{div}u\|_{L^2}+\|\mbox{curl}^2u\|_{L^2}+\|\nabla u\|_{L^2}\right),
	\end{align*}
	which implies that $\|\nabla^2u\|_{L^2_t({L^2})}$ can be controlled by Lemma \ref{lem-basic}, Lemma \ref{lem-div} and Lemma \ref{lem-w}.
	Similarly,
	\begin{align}\label{2ut}
		\|\nabla^2 u_t\|_{L^2}\leq C\left(\|\nabla\mbox{div}u_t\|_{L^2}+\|\mbox{curl}^2u_t\|_{L^2}+\|\nabla u_t\|_{L^2}\right),
	\end{align}
	and
	\begin{align*}
		\|\nabla^3 u\|_{L^2}\leq C \left(\| \nabla^2 \mbox{div} u\|_{L^2}+\|\mbox{curl}^3u\|_{L^2}+\|\nabla^2 u\|_{L^2}+\|\nabla u\|_{L^2}\right),
	\end{align*}
	with all the terms on the right-hand side already have estimates.
	Then we integrate with respect to time $t$ and choose small enough $\varepsilon>0$
	\begin{align} \label{almost}
		\mathcal{E}^2(t) + c \int_{0}^{t} \mathcal{D}^2(s) \mathrm{d} s & \leq C\mathcal{E}^2(0) + C \delta \int_{0}^{t} \mathcal{D}^2(s) \mathrm{d} s \nonumber\\[2mm]
		&- c_1 \int_{\Omega} ( u \cdot \nabla (h'(\tilde{ \rho}) q) + \tilde{ \rho}^{-1} h''(\tilde{ \rho}) \nabla q_t \cdot \nabla \tilde{ \rho} q_t + \mbox{div} u_t \mbox{div} u \nonumber\\[2mm]
		&+ \tilde{ \rho}^{-1} q \nabla^2 q \cdot (h''(\tilde{ \rho}) \nabla^2 \tilde{ \rho} + h'''(\tilde{ \rho})\nabla \tilde{ \rho} \nabla \tilde{ \rho}) + \tilde{ \rho}^{-1} h''(\tilde{ \rho}) \nabla \tilde{ \rho} \cdot \nabla^2 q \cdot \nabla q ) ((t)-(0)) \nonumber\\[2mm]
		&+ \int_{\partial\Omega} ((S(n)- \alpha I)u_t \cdot u_t)((t)-(0)) \mathrm{d} \sigma,
	\end{align}
	where $c_1>0$ is a small constant that can be chosen later to make these terms controlled by the left-hand side of the (\ref{almost}).
	
	We apply the H$\ddot{o}$lder's inequalities and Cauchy's inequalities directly
	\begin{align*}
		&- c_1 \int_{\Omega} ( u \cdot \nabla (h'(\tilde{ \rho}) q) + \tilde{ \rho}^{-1} h''(\tilde{ \rho}) \nabla q_t \cdot \nabla \tilde{ \rho} q_t + \mbox{div} u_t \mbox{div} u \nonumber\\[2mm]
		&+ \tilde{ \rho}^{-1} q \nabla^2 q \cdot (h''(\tilde{ \rho}) \nabla^2 \tilde{ \rho} + h'''(\tilde{ \rho})\nabla \tilde{ \rho} \nabla \tilde{ \rho}) + \tilde{ \rho}^{-1} h''(\tilde{ \rho}) \nabla \tilde{ \rho} \cdot \nabla^2 q \cdot \nabla q ) ((t)-(0))\\[2mm]
		&\leq c_1 C( \| \nabla u (t) \|^2_{L^2} + \|q(t)\|^2_{H^2} + \|\nabla u_t(t)\|^2_{L^2} + \|q_t(t)\|^2_{H^1}) \nonumber\\[2mm]
		&+c_1 C( \| \nabla u (0)\|^2_{L^2} + \|q(0)\|^2_{H^2} + \|\nabla u_t(0)\|^2_{L^2} + \|q_t(0)\|^2_{H^1})\nonumber \\[2mm]
		&\leq c_1 C \mathcal{E}^2(t) + c_1 C \mathcal{E}^2(0).
	\end{align*}
	Noted that $-S(n)$ is positive semi-definite in our setting $\Omega\equiv\mathbb{R}^3\setminus B_R$  (see Remark \ref{rem1}), then the above boundary integrals $\int_{\partial\Omega}S(n)u_t\cdot u_t-\int_{\partial\Omega} \alpha^{+}|u_t|^2$ is good terms in the estimate since they are negative terms on the right-hand side. Then we shall use the trace theorem to estimate the boundary integral,
	\begin{align*}
		-\int_{\partial\Omega} \alpha^{-} |u_t|^2\leq |\alpha^{-}|_{L^{\infty}}|u_t|^2_{L^{2}(\partial\Omega)}\leq C^{*}|\alpha^{-}|_{L^{\infty}} R \|\nabla u_t\|^2 \leq C^{*}|\alpha^{-}|_{L^{\infty}} R \mathcal{E}^2(t),
	\end{align*}
	where we used Proposition \ref{prop-trace}.
	By choosing $c_1$ and $R$ small enough, we complete the proof of Proposition \ref{prop}.
	\qed

	\centerline{\bf Acknowledgements}
	This research is supported by National Natural Science
	Foundation of China (No. 12471198, No. 12431018). The authors would like to thank Professor
	Tao Luo  for his interests in this work.

\end{document}